\documentclass[final,3p,times,12pt]{elsarticle}
\usepackage{multicol}        
\usepackage[bottom]{footmisc}
\usepackage{mdframed}

\usepackage[english]{babel}
\usepackage{amssymb,amstext,amsmath,amsthm,mathabx}
\usepackage{hyperref}
\usepackage{array}   
\newcolumntype{L}{>{$}l<{$}}
\theoremstyle{definition}
\newtheorem{definition}{Definition} 
\newtheorem{remark}{Remark} 
\newtheorem{theorem}{Theorem}

\newcommand{\no}[1]{\widebar{#1}}

\def\F{\mathcal F}
\def\P{\mathcal P}
\def\M{\mathcal M}
\def\I{\mathcal I}
\def\B{\mathcal B}
\def\H{\mathcal H}
\def\pr{\mathbb{P}}
\def\prev{\mathbb{P}}
\def\G{\mathcal{G}}
\usepackage[mathscr]{euscript}
\def\C{\mathscr{C}}
\def\K{\mathscr{K}}
\def\D{\mathscr{D}}
\def\Q{\mathcal{Q}}
\newtheorem{example}{Example}

\usepackage{amssymb}

\journal{International Journal of Approximate Reasoning, }

\begin{document}

\begin{frontmatter}





\title{Algebraic aspects and coherence conditions for conjoined and disjoined conditionals
 \tnoteref{mytitlenote}
}
\author[ag]{Angelo Gilio\fnref{fn1,fn2}}
\address[ag]{Department of Basic and Applied Sciences for Engineering, University of Rome ``La Sapienza'', Via A. Scarpa 14, 00161 Roma, Italy}
\ead{angelo.gilio@sbai.uniroma1.it}

\cortext[cor1]{Corresponding author}

\author[gs]{Giuseppe Sanfilippo\fnref{fn1,fn3}\corref{cor1}}
\address[gs]{Department of Mathematics and Computer Science, Via Archirafi 34, 90123 Palermo, Italy}
\ead{giuseppe.sanfilippo@unipa.it}
\tnotetext[myttitlenote]{Int. J. Approx. Reason., \url{https://doi.org/10.1016/j.ijar.2020.08.004}}
\fntext[fn1]{Both authors  equally contributed to this work}
\fntext[fn2]{Retired}
\fntext[fn3]{Also affiliated with INdAM-GNAMPA, Italy}
\begin{abstract}
We deepen the study of conjoined and disjoined  conditional events  in the setting of coherence. These objects, differently from other approaches, are defined  in the framework of conditional random quantities. We show that some well known properties, valid in the case of unconditional events, still hold in our approach to logical operations among  conditional events.  In particular we prove a decomposition formula and a related additive property. Then,  we introduce the set of conditional constituents generated by $n$ conditional events and we  show that they satisfy the basic properties valid in the case of unconditional events. We obtain a generalized inclusion-exclusion formula and we prove a suitable distributivity property. Moreover, under logical independence of basic unconditional events, we give two necessary and sufficient  coherence conditions.
The first condition gives a  geometrical  characterization for
 the coherence of  prevision assessments  on a family $\mathscr{F}$ constituted by   $n$ conditional events and  all  possible conjunctions among  them. 
The second condition characterizes the coherence of prevision assessments defined on $\mathscr{F}\cup \mathscr{K}$, where $\mathscr{K}$ is the set of conditional constituents associated with the conditional events in $\mathscr{F}$.
Then, we give  a further theoretical result and we examine some examples and counterexamples.
 Finally, we make a comparison with other approaches and we illustrate some theoretical aspects and applications.
\end{abstract}
\begin{keyword}
Coherence  \sep 
Conditional random quantities \sep Conjunction and disjunction of conditionals \sep  Decomposition formula \sep Conditional constituents \sep Inclusion-exclusion formula.
\end{keyword}
\end{frontmatter}
\section{Introduction and motivations}
The study of logical operations among conditional events is a relevant topic of research in many fields, such as  probability logic, multi-valued logic, artificial intelligence, and psychology of reasoning; it has been largely discussed and investigated by many authors (see, e.g.,
 \cite{Baratgin18,benferhat97,CoSV13,CoSV15,Douven19,FlGH20,GoNW91,Kauf09,McGe89,NgWa94}). We recall that in a pioneering paper, written in 1935, de Finetti (\cite{defi36}) proposed  a three-valued logic for conditional events, also studied by  Lukasiewicz. Moreover, different authors (such as Adams, Belnap, Calabrese, de Finetti, Dubois, van Fraassen, McGee, Goodmann, Lewis, Nguyen, Prade, Schay) 
 have  given many contributions to research on  three-valued logics and compounds of conditionals (for a survey see, e.g.,\cite{Miln97}).
Conditionals have been extensively studied also in \cite{edgington95,McGe89}.

Usually,  the result of the conjunction or the disjunction  of conditionals, as defined in literature, is still a conditional; see e.g.   \cite{adams75,Cala87,Cala17,CiDu12,CiDu13,GoNW91}. However, in this way classical probabilistic properties are lost; for instance, differently from the case of unconditional events, the lower and upper probability bounds for the conjunction of two conditional events are no more the Fr\'echet-Hoeffding bounds; in some cases trivially these bounds are 0 and 1, respectively. This aspect has been recently studied in \cite{SUM2018S}.

A different approach, where the result of conjunction or disjunction of conditionals is not
a  three-valued object, has been given in  \cite{Kauf09,McGe89}.  
In \cite{GiSa13c,GiSa13a,GiSa14}
  a related theory
has been  developed in the setting of coherence,
with the advantage (among other things) of properly managing the case where some conditioning events
have zero probability.
In these papers, the results of conjunction and disjunction of conditional events
are  {\em conditional random quantities} with a finite number of  possible values in the interval $[0,1]$. 

In addition, it has been proved that the  Fr\'echet-Hoeffding probability bounds continue to hold for the conjunction of two conditional events (\cite{GiSa14}).  In this paper, 
we give a related result which concerns the conjunctions 
associated with
 two disjoint sub-families  of a family of $n$ conditional events.

We show that the conjunction $\C_{1\cdots n}$ of $n$ conditional events can be decomposed as the sum of its conjunctions with a further conditional event $E_{n+1}|H_{n+1}$ and the negation $\no{E}_{n+1}|H_{n+1}$. This result generalizes the well known formula (for the indicators of two unconditional events $A$ and $B$): $A=AB+A\no{B}$.

We give a generalization of the inclusion-exclusion formula for the disjunction of a finite number of conditional events. 
Moreover, we prove the validity of a suitable distributivity property, by means of which we can  directly obtain the inclusion-exclusion formula.
A main motivation  of the  paper is that of  
introducing 
the  conditional constituents for a finite family of conditional events, which
can be looked at as a conditional counterpart of  atoms of a Boolean algebra.  We show that the  conditional constituents  satisfy the basic numerical and probabilistic properties of the (indicators of the) constituents associated with a finite family of unconditional events.  

   Under  logical independence,
 we give  a necessary and sufficient condition for coherence of a prevision assessment $\M$ on a  family $\F$  containing $n$ conditional events and  all the ($2^n-n-1$) possible  conjunctions among them. Such a  characterization  amounts to the  solvability of a linear system which can be interpreted in geometrical terms.
Then,  the set of all coherent assessments  on the  family $\mathscr{F}$ is represented by a list of linear inequalities on the components of each  prevision assessment $\M$. 

In addition, by considering the set $\mathscr{K}$ of conditional constituents associated with the conditional events in $\mathscr{F}$, we give a result which under logical independence characterizes the coherence of prevision assessments on $\mathscr{F} \cup \mathscr{K}$.
Then, given any  coherent assessment  $\M$ on $\F$, we  show that every  possible value of the random vector associated with $\mathscr{F}$ is itself a particular coherent assessment on $\mathscr{F}$. 
To better illustrate  our results, we examine some examples and counterexamples.

Finally, we make  a comparison with other approaches, by giving a result related to the notion of atom of a Boolean algebra of conditionals  introduced in \cite{FlGH17,FlGH20}. In this context,  we discuss the significance of our theory by recalling some theoretical aspects and applications.

The paper is organized as follows:
In Section \ref{SECT:PRELIMINARIES} we recall some basic notions and results on coherence of conditional probability   and prevision assessments. We  also recall the definition of conjunction and disjunction among conditional events, and the notion of negation.
In Section  \ref{SEC:DECOMP} we first give  a result related to   Fr\'echet-Hoeffding  bounds; then we illustrate the decomposition formula for the conjunction of $n$ conditional events.
In Section \ref{SEC:CONDCONST} we introduce the set $\mathscr{K}$ of conditional constituents for a  family   of  $n$ conditional events $\mathscr{E}=\{E_1|H_1,\ldots, E_n|H_n\}$. We show that, as in the case of unconditional events, the sum of the conditional constituents is equal to 1 and for each pair of them the conjunction is equal to 0. Then we show that, for each non empty subset $S\subseteq\{1,\ldots,n\}$   the conjunction  $\C_{S}=\bigwedge_{i\in S}E_i|H_i$  is the sum of suitable conditional constituents in $\mathscr{K}$; hence  the prevision of $\C_S$ is the sum of the previsions of such conditional constituents.  
In Section  \ref{SEC:INEX}   we give a generalization of the inclusion-exclusion formula for the disjunction of $n$ conditional events.
Then, we prove  a suitable distributivity property and we examine  related  probabilistic results. 
In Section \ref{SEC:NECSUFF}, under the hypothesis of logical independence of basic unconditional events, 
we characterize in terms of a suitable convex hull the set of all coherent prevision assessments on a family  $\mathscr{F}$ containing $n$ conditional events and all the possible conjunctions among them. Such a  characterization  amounts to the  solvability of a linear system. Then, we illustrate the set of all coherent assessments on the  family $\mathscr{F}$ by  a list of linear inequalities on the components of each prevision assessment. We also characterize the coherence of prevision assessments on $\mathscr{F} \cup \mathscr{K}$.
In Section \ref{EQ:FURTHASP}, given any  coherent assessment  $\M$ on $\F$,
we 
show that every  possible value of the random vector associated with $\mathscr{F}$ is itself a particular coherent assessment on $\mathscr{F}$.
In Section \ref{SEC:EX} we illustrate further  aspects on coherence by examining  some examples and counterexamples.
In Section \ref{SEC:COMP}, after a comparison with other approaches,
we give a result related to the notion of atom of a Boolean algebra of conditionals  introduced in \cite{FlGH17,FlGH20} and we illustrate some theoretical aspects and applications of our theory.
In Section \ref{SEC:CONC} we give some conclusions.

\section{Some preliminary notions and results} \label{SECT:PRELIMINARIES}
In this section we recall some basic notions and results which concern coherence (see, e.g., \cite{BeMR17,biazzo05,BiGS12,CaLS07,coletti02,PeVa17,PfSa17}) and logical operations among conditional events (see \cite{GiSa13c,GiSa13a,GiSa14,GiSa17,GiSa19}).
\subsection{Events and conditional events}
\label{sect:2.2}
An event $E$ is an uncertain fact described by a (non ambiguous) logical proposition; in formal terms $E$  is a two-valued logical entity which can be \emph{true}, or \emph{false}.
The \emph{indicator} of $E$, denoted by the same symbol, is  1, or 0, according to  whether $E$ is true, or false. The sure event and impossible event are denoted by $\Omega$ and  $\emptyset$, respectively. 
Given two events $E_1$ and $E_2$,  we denote by $E_1\land E_2$, or simply by $E_1E_2$, (resp., $E_1 \vee E_2$) the logical conjunction (resp., the  logical disjunction).   The  negation of $E$ is denoted $\no{E}$.  We simply write $E_1 \subseteq E_2$ to denote that $E_1$ logically implies $E_2$, that is  $E_1\no{E}_2=\emptyset$. 
We recall that  $n$ events $E_1,\ldots,E_n$ are logically independent when the number $m$ of constituents, or possible worlds, generated by them  is $2^n$ (in general  $m\leq 2^n$).

Given two events $E,H$,
with $H \neq \emptyset$, the conditional event $E|H$
is defined as a three-valued logical entity which is \emph{true}, or
\emph{false}, or \emph{void}, according to whether $EH$ is true, or $\no{E}H$
is true, or $\no{H}$ is true, respectively. Given a family $\mathscr{E} = \{E_1|H_1, \ldots, E_n|H_n\}$, 
we observe that, for each $i$, it holds that $E_iH_i \vee \no{E}_iH_i
\vee \no{H}_i=\Omega$; then by expanding the expression $\bigwedge_{i=1}^n(E_iH_i \vee \no{E}_iH_i
\vee \no{H}_i)$ we can represent $\Omega$ as the disjunction of $3^n$
logical conjunctions, some of which may be impossible. 
The
remaining ones are the constituents generated by 
$\mathscr{E}$ and, of course, are a partition of $\Omega$.
We denote by $C_1, \ldots, C_m$ the constituents which logically imply  the event
$\mathcal{H}_n=H_1 \vee \cdots \vee H_n$. Moreover,
 (if $\mathcal{H}_n \neq \Omega$) we denote by $C_0$ the
remaining constituent $\mathcal{\no{H}}_n = \no{H}_1 \cdots \no{H}_n$. Thus
\[
\begin{array}{ll}
\mathcal{H}_n = C_1 \vee \cdots \vee C_m \,,\;\;\; \Omega =
\mathcal{\no{H}}_n \vee
\mathcal{H}_n = C_0 \vee C_1 \vee \cdots \vee C_m \,,\;\;\; m+1 \leq 3^n
\,.
\end{array}
\]
For instance, given four logically independent events $E_1, E_2, H_1, H_2$, the constituents generated by $\mathscr{E}=\{E_1|H_1,E_2|H_2\}$ are $C_1=E_1H_1E_2H_2$, $C_2=E_1H_1\no{E}_2H_2$,
$C_{3}=\no{E}_1H_1  E_2H_2 $,         
$C_{4}=\no{E}_1H_1  \no{E}_2H_2$,     
$C_{5}=\no{H}_1     E_2H_2$,
$C_{6}=\no{H}_1     \no{E}_2H_2$,     
$C_{7}=E_1H_1  \no{H}_2$,
$C_{8}=\no{E}_1H_1  \no{H}_2$,
$C_0=\no{H}_1     \no{H}_2$.
\subsection{Coherent conditional prevision assessments for conditional random quantities}
\label{Coherence}
Given a  (real) random quantity $X$ and an event $H\neq \emptyset$, we denote by  $\prev(X|H)$ the prevision of $X$ conditional on $H$. In the framework of coherence, to assess $\prev(X|H)=\mu$ means that, for every real number $s$,  you are willing to pay 
an amount $s\mu$ and to receive  $sX$, or $s\mu$, according
to whether $H$ is true, or $\widebar{H}$
is true (the bet is called off), respectively. 
The random gain is $G=s(XH+\mu \widebar{H})-s\mu=
sH(X-\mu)$. 

As we will see, a  conjunction of $n$ conditional events is a conditional random quantity with a finite number of possible (numerical) values. Then,
in what follows,
for any given conditional random quantity $X|H$, we assume that, when $H$ is true, the set of possible values of $X$ is a finite subset of  the set of real numbers $\mathbb{R}$.
 In this case we say 
that $X|H$ is a finite conditional random quantity.  
Given a prevision function $\pr$ defined on an arbitrary family $\mathcal{K}$ of finite
conditional random quantities, consider a finite subfamily $\F = \{X_1|H_1, \ldots,X_n|H_n\} \subseteq \mathcal{K}$ and the vector
$\M=(\mu_1,\ldots, \mu_n)$, where $\mu_i = \pr(X_i|H_i)$ is the
assessed prevision for the conditional random quantity $X_i|H_i$, $i\in \{1,\ldots,n\}$.
With the pair $(\F,\M)$ we associate the random gain $G =
\sum_{i=1}^ns_iH_i(X_i - \mu_i)$. We  denote by $\G_{\mathcal{H}_n}$ the set of values of $G$ restricted to $\H_n= H_1 \vee \cdots \vee H_n$. 
Then, by the {\em  betting scheme} of de Finetti, the notion of coherence is defined as below.

\begin{definition}\label{COER-RQ}
		The prevision function $\pr$ defined on $\mathcal{K}$ is coherent if and only if, $\forall n
		\geq 1$,  $\forall \, \F=\{X_1|H_1, \ldots,X_n|H_n\} \subseteq \mathcal{K}$, it holds that: $min \; \G_{\mathcal{H}_n} \; \leq 0 \leq max \;
		\G_{\mathcal{H}_n}$, $\forall \, s_1, \ldots,
		s_n$. 
\end{definition}
A conditional prevision assessment $\prev$ on $\mathcal{K}$ is said  incoherent 	if and only if  there exists a finite combination of $n$ bets such that $ \min\mathcal{G}_{\mathcal{H}_n}\cdot \max \mathcal{G}_{\mathcal{H}_n}>0$, that is such that
the values in $\mathcal{G}_{\mathcal{H}_n}$
 are all positive, or all negative ({\em Dutch Book}). In the particular case where $\mathcal{K}$ is a family of conditional events, then Definition \ref{COER-RQ} becomes  the well known definition of coherence for a probability function $\prev$, denoted  as $P$, defined on $\mathcal{K}$.

Given a family $\F = \{X_1|H_1,\ldots,X_n|H_n\}$, for each $i = 1,\ldots,n,$ we denote by $\{x_{i1}, \ldots,x_{ir_i}\}$ the set of possible (numerical) values for the restriction of $X_i$ to $H_i$; then, for each $i = 1,\ldots,n,$ and $j = 1, \ldots, r_i$, we set $A_{ij} = (X_i = x_{ij})$. Of course, for each  $i$,  the family $\{\no{H}_i,\, A_{ij}H_i \,,\; j = 1, \ldots, r_i\}$ is a partition of the sure event $\Omega$, with  $A_{ij}H_i=A_{ij}$ and          $\bigvee_{j=1}^{r_i}A_{ij}=H_i$. Then,
the constituents generated by the family $\F$ are (the
elements of the partition of $\Omega$) obtained by expanding the
expression $\bigwedge_{i =1}^n(A_{i1} \vee \cdots \vee A_{ir_i} \vee
\no{H}_i)$. We set $C_0 = \no{H}_1 \cdots \no{H}_n=\no{\H}_n$ (it may be $C_0 = \emptyset$);
moreover, we denote by $C_1, \ldots, C_m$ the constituents
contained in $\H_n$. Hence
$\bigwedge_{i=1}^n(A_{i1} \vee \cdots \vee A_{ir_i} \vee
\no{H}_i) = \bigvee_{h = 0}^m C_h$.
With each $C_h,\, h =1,\ldots,m$, we associate a vector
$Q_h=(q_{h1},\ldots,q_{hn})$, where $q_{hi}=x_{ij}$ if $C_h \subseteq
A_{ij},\, j=1,\ldots,r_i$, while $q_{hi}=\mu_i$ if $C_h \subseteq \no{H}_i$;
with $C_0$ it is associated  $Q_0=\M = (\mu_1,\ldots,\mu_n)$. As, for each $i, j$, the quantities $x_{ij}, \mu_i$  are real numbers, it holds that  $Q_h\in \mathbb{R}^n$, $h=0,1,\ldots, m$.

Denoting by $\I$ the convex hull of $Q_1, \ldots, Q_m$, the condition  $\M\in \I$ amounts to the existence of a vector $(\lambda_1,\ldots,\lambda_m)$ such that:
$ \sum_{h=1}^m \lambda_h Q_h = \M \,,\; \sum_{h=1}^m\lambda_h
= 1 \,,\; \lambda_h \geq 0 \,,\; \forall \, h$; in other words, $\M\in \I$ is equivalent to the solvability of the system $(\Sigma)$, associated with  $(\F,\M)$, given below.
\begin{equation}\label{SYST-SIGMA}
(\Sigma) 
\left\{
\begin{array}{ll}
\sum_{h=1}^m \lambda_h q_{hi} =
\mu_i \,,\; i =1,\ldots,n, \\
\sum_{h=1}^m \lambda_h = 1,\;\;
\lambda_h \geq 0 \,,\;  \, h=1,\ldots,m.
\end{array}
\right.
\end{equation}
Given the assessment $\M =(\mu_1,\ldots,\mu_n)$ on  $\F =
\{X_1|H_1,\ldots,X_n|H_n\}$, let $S$ be the set of solutions $\Lambda = (\lambda_1, \ldots,\lambda_m)$ of system $(\Sigma)$ defined in  (\ref{SYST-SIGMA}).   
Then, the  following characterization theorem for coherent assessments on  finite families of conditional random quantities can be proved (\cite{BiGS08}).
\begin{theorem}\label{SYSTEM-SOLV}{ \rm [{\em Characterization of coherence}].
		Given a family of $n$ conditional random quantities $\F =
		\{X_1|H_1,\ldots,X_n|H_n\}$, with finite sets of possible values, and a vector $\M =
		(\mu_1,\ldots,\mu_n)$, the conditional prevision assessment
		$\prev(X_1|H_1) = \mu_1 \,,\, \ldots \,,\, \prev(X_n|H_n) =
		\mu_n$ is coherent if and only if, for every subset $J \subseteq \{1,\ldots,n\}$,
		defining $\F_J = \{X_i|H_i \,,\, i \in J\}$, $\M_J = (\mu_i \,,\,
		i \in J)$, the system $(\Sigma_J)$ associated with the pair
		$(\F_J,\M_J)$ is solvable. }
	\end{theorem}
	As shown by 
Theorem~\ref{SYSTEM-SOLV},  the solvability of system $(\Sigma)$ (i.e., the condition
$\mathcal{M}\in \mathcal{I}$) is a necessary (but not sufficient) condition
for coherence of $\mathcal{M}$ on $\mathcal{F}$.
Given the assessment $\mathcal{M}$ on
$\mathcal{F}$, let $S$ be the set of
solutions $\Lambda = (\lambda _{1}, \ldots ,\lambda _{m})$ of system
$(\Sigma )$ defined in (\ref{SYST-SIGMA}). By assuming the system
$(\Sigma )$ solvable, that is $S \neq \emptyset $, we define:
%
\begin{equation}\label{EQ:I0}
\begin{array}{ll}
I_0 = \{i : \max_{\Lambda \in S}  \sum_{h:C_h\subseteq H_i}\lambda_h= 0\},\;\;
  \F_0 = \{X_i|H_i \,, i \in I_0\},\;\;  \M_0 = (\mu_i ,\, i \in I_0)\,.
\end{array}
\end{equation}
We observe that   $i\in I_0$ if and only if the (unique) coherent extension of $\M$ to $H_i|\H_n$ is zero.
Then, the following theorem can be proved  (\cite[Theorem 3]{BiGS08}):
\begin{theorem}\label{CNES-PREV-I_0-INT}{\rm [{\em Operative characterization of coherence}]
		A conditional prevision assessment ${\M} = (\mu_1,\ldots,\mu_n)$ on
		the family $\F = \{X_1|H_1,\ldots,X_n|H_n\}$ is coherent if
		and only if the following conditions are satisfied: \\
		$(i)$ the system $(\Sigma)$ defined in (\ref{SYST-SIGMA}) is solvable; $(ii)$ if $I_0 \neq \emptyset$, then $\M_0$ is coherent. }
\end{theorem}

In order to illustrate the previous results, we examine an example. 
\begin{example}
Let  $E, H, K$ be three events, with 
	$HK=\emptyset$ and $E$ logically independent of $H$ and $K$. Moreover, let $\P=(x,y)$ be a probability  assessment on the family $\mathscr{E}=\{E|H,E|K\}$, where $x=P(E|H)$ and $y=P(E|K)$. The constituents generated by $\mathscr{E}$ are: 
	$C_1=EH\no{K}$,
	$C_2=E\no{H}K$, $C_3=\no{E}H\no{K}$, $C_4=\no{E}\no{H}K$, $C_0=\no{H}\no{K}$.
	Then, the associated points $Q_h$'s  are: $
	Q_1=(1,y)$, $Q_2=(x,1)$, $Q_3=(0,y)$, $Q_4=(x,0)$, $Q_0=\P=(x,y)$.
The system $(\Sigma)$ is
\[
\left\{
\begin{array}{ll}
 \lambda_1 +\lambda_2\,x +\lambda_4\,x \,\,= \,\,x ,\;\;\lambda_1\,y +\lambda_2 +\lambda_3\,y\,\,= \,\,y ,\,\,
 \lambda_1 +\lambda_2 +\lambda_3+\lambda_4\,\,= \,\,1 ,\;\; \lambda_h\geq 0, \;\;\forall h.\\
\end{array}
\right.
\]
As it can be verified, for each $(x,y)\in[0,1]^2$, the vector
$(\lambda_1,\ldots,\lambda_4)=(\frac{x}{2},\frac{y}{2},\frac{1-x}{2},\frac{1-y}{2})$  is a solution of $(\Sigma)$. Moreover,  for this solution it holds that 
\[
\sum_{C_h\subseteq H}\lambda_h=\lambda_1+\lambda_3=\frac{1}{2}>0;\;\;\; \sum_{C_h\subseteq K}\lambda_h=\lambda_2+\lambda_4=\frac{1}{2}>0.
\]
 Then, $I_0=\emptyset$ and by Theorem \ref{CNES-PREV-I_0-INT} the  assessment $(x,y)$ is coherent, for every $(x,y)\in[0,1]^2$.
\end{example}
\subsection{A deepening on the notion of conditional random quantity}
We recall that, in the subjective approach to probability theory, given an event $H \neq \emptyset$ and a random quantity $X$ by the betting metaphor the conditional prevision  $\pr(X|H)$ is defined as the amount $\mu$ you agree to pay, by knowing that you will receive the amount $XH+\mu\no{H}$. This quantity
coincides with $X$, if $H$ is true, or with $\mu$, if $H$ is false  (bet  called off). 
Usually, in literature  the conditional random quantity $X|H$ is  defined as the {\em restriction} of $X$ to $H$, which coincides with  $X$, when $H$ is true, and it  is undefined when $H$ is false. Under this point of view, (when $H$ is false) $X|H$ does not coincide with $XH+\mu\no{H}$. However, by
coherence,  it holds that  $\pr(XH+\mu\no{H})=
\pr(X|H)P(H)+\mu P(\no{H})=\mu P(H)+\mu P(\no{H})=  \mu$. Then, we can extend the notion of $X|H$, by  defining  its value as equal to $\mu$ when  $H$ is false (for further details see \cite{GiSa14}).
In this way $X|H$ coincides with $XH+\mu\no{H}$ and   in the betting scheme it can be interpreted as the amount that you receive when you pay its prevision $\mu$.  In addition, the random gain $G$ can be  represented as $G=s(X|H - \mu)$.
In particular,  when $X$ is the indicator of an event $E$, we obtain 
$X|H=EH+P(E|H)\no{H}$ and it holds that 
\[
\prev(X|H)=\prev[EH+P(E|H)\no{H}]
=P(E|H)P(H)+P(E|H)P(\no{H})=P(E|H).
\]
In this case $X|H$ is the  indicator of the conditional event  $E|H$  (which we denote  by the same symbol)  and, 
 by defining $P(E|H)=x$, 
it holds that
\begin{equation}\label{EQ:AgH}
E|H=
EH+x \widebar{H}=EH+x (1-H)=\left\{\begin{array}{ll}
1, &\mbox{if $EH$ is true,}\\
0, &\mbox{if $\no{E}H$ is true,}\\
x, &\mbox{if $\no{H}$ is true.}\\
\end{array}
\right.
\end{equation} 
For related discussions, see also \cite{CoSc99,GiSa13c,lad96}.
By Definition \ref{COER-RQ},  the coherence of the assessment $P(E|H)=x$ is equivalent to $min \; \G_{H} \; \leq 0 \leq max \;
\G_{H}$, $\forall \, s$, where $\G_{H}$ is the set of values of $G$ restricted to $H$. Then, the set $\Pi$ of coherent assessments $x$ on $E|H$ is: $(i)$ $\Pi=[0,1]$, when $\emptyset \neq EH \neq H$;
$(ii)$ $\Pi=\{0\}$, when $ EH=\emptyset$; $(iii)$ $\Pi=\{1\}$, when $ EH=H$.  Of course, the third  value of the random quantity  $E|H$  depends on the subjective assessment  $P(E|H)=x$.
Notice that,  when $H\subseteq E$ (i.e., $EH=H$), by coherence $P(E|H)=1$ and hence for the indicator it holds that $E|H=H+\no{H}=1$. 

By exploiting our extended notion of conditional random quantity, we can develop some algebraic aspects (\cite[Section 3]{GiSa13c},\cite[Section 3.2]{GiSa14}). For instance, we can show that:
	\begin{itemize}
	\item 
	denoting by $\mu$ and $\nu$ the previsions of 
	$X|H$ and $Y|K$, respectively, the sum $X|H+Y|K$ coincides with the conditional random quantity 
	$(XH+\mu \overline{H}+YK+\nu \overline{K})|(H \vee K)$, with
	$\mathbb{P}(X|H+Y|K)=\mathbb{P}(X|H)+\mathbb{P}(Y|K)=\mu +\nu $;
	\item $a(X|H) + b(Y |K)=(aX)|H + (bY )|K $, where $a,b$ are real numbers;
	\item $\mathbb{P}(XH|K)=P(H|K)\mathbb{P}(X|HK)$, which is the
	\emph{compound prevision theorem}.
\end{itemize}
Moreover, as shown by the result below, if  $X|H$ and $Y|K$ coincide when $H \vee K$ is true, then their previsions are equal and it follows that $X|H$ and $Y|K$ also coincide when $H \vee K$ is false, so that $X|H = Y|K$ in all cases (\cite[Theorem 4]{GiSa14}).
\begin{theorem}\label{THM:EQ-CRQ} Given any events $H\neq \emptyset$, $K\neq \emptyset$, and any r.q.'s $X$, $Y$, let $\Pi$ be the set of the coherent prevision assessments $\pr(X|H)=\mu,\pr(Y|K)=\nu$. \\
		(i) Assume that, for every $(\mu,\nu)\in \Pi$,  $X|H=Y|K$  when  $H\vee K$ is true; then   $\mu= \nu$ for every $(\mu,\nu)\in \Pi$. \\
		(ii) For every $(\mu,\nu)\in \Pi$,   $X|H=Y|K$  when  $H\vee K$ is true  if and only if $X|H= Y|K$.
\end{theorem}
To better illustrate Theorem \ref{THM:EQ-CRQ}, we observe that
\[
X|H-Y|K=(XH+\mu \no{H}-YK-\nu \no{K})|(H \vee K).
\] 
	Now, assume that $X|H$ and $Y|K$ coincide when $H\vee K$ is true, so that
$X|H-Y|K=0$ when $H\vee K$ is true. If $H\vee K$ is false, that is
$\overline{H}\,\overline{K}$ is true, it holds that $X|H=\mu $ and
$Y|K=\nu $, so that $X|H-Y|K=\mu-\nu$. Then, in a conditional bet on the conditional random quantity
$X|H-Y|K$, if you pay $\mu -\nu =\mathbb{P}(X|H-Y|K)$, you receive zero
when $H\vee K$ is true, or you receive back $\mu -\nu $ when
$H\vee K$ is false (bet called off). Then, by coherence, it must be
$\mu -\nu =0$, that is $\mu =\nu $, and hence $X|H=Y|K$.
%
\begin{remark}\label{REM:INEQ-CRQ}
	Theorem \ref{THM:EQ-CRQ} has been generalized in \cite[Theorem 6]{GiSa19} by replacing the symbol ``$=$'' by ``$\leq$'' in statements  $(i)$ and $(ii)$. In other words,
	if  $X|H\leq Y|K$ when  $H\vee K$ is true, then $\prev(X|H)\leq  \prev(Y|K)$ and hence  $X|H\leq  Y|K$ in all cases.	
\end{remark}		
\subsection{Logical operations among conditional events}
We recall below the notions of conjunction and disjunction of two conditional events.
\begin{definition}\label{CONJUNCTION} Given any pair of conditional events $E_1|H_1$ and $E_2|H_2$, with $P(E_1|H_1)=x_1$ and $P(E_2|H_2)=x_2$, their conjunction $(E_1|H_1) \wedge (E_2|H_2)$ is the conditional random quantity defined as
	\begin{equation}\label{EQ:CONJUNCTION}
	\begin{array}{lll}
	(E_1|H_1) \wedge (E_2|H_2)&=&(E_1H_1E_2H_2+x_{1}\no{H}_1E_2H_2+x_{2}\no{H}_2E_1H_1)|(H_1\vee H_2) =\\
	&=&\left\{\begin{array}{ll}
	1, &\mbox{if $E_1H_1E_2H_2$ is true,}\\
	0, &\mbox{if $\no{E}_1H_1\vee \no{E}_2H_2$ is true,}\\
	x_1, &\mbox{if $\no{H}_1E_2H_2$ is true,}\\
	x_2, &\mbox{if $\no{H}_2E_1H_1$ is true,}\\
	x_{12}, &\mbox{if $\no{H}_1\no{H}_2$ is true},
	\end{array}
	\right.
	\end{array}
	\end{equation}		
	where $x_{12}=\prev[(E_1|H_1) \wedge (E_2|H_2)]=\prev[(E_1H_1E_2H_2+x_{1}\no{H}_1E_2H_2+x_{2}\no{H}_2E_1H_1)|(H_1\vee H_2)]$.
\end{definition}
In betting terms, the prevision $x_{12}$ represents the amount you agree to pay, with the proviso that you will receive the quantity $E_1H_1E_2H_2+x_{1}\no{H}_1E_2H_2+x_{2}\no{H}_2E_1H_1$, or you will receive back the quantity $x_{12}$, according to whether $H_1\vee H_2$ is true, or  $\no{H}_1\no{H}_2$ is true. In other words, by paying $x_{12}$ you receive
 $E_1H_1E_2H_2+x_{1}\no{H}_1E_2H_2+x_{2}\no{H}_2E_1H_1+x_{12}\no{H}_1\no{H}_2$, which assumes one of the following values: \vspace{-0.3cm}
\begin{itemize}
	\item $1$, if both conditional events are true;\vspace{-0.3cm}
	\item  $0$, if at least one of the conditional events is false; \vspace{-0.3cm}
	\item  the probability of  the conditional event that is void if one conditional event is void  and the other one is true;\vspace{-0.3cm}
	\item  $x_{12}$ (the amount that you paid) if both 
	conditional events are void.
\end{itemize}
\begin{remark}\label{REM:CIRC}
By recalling  (\ref{EQ:AgH}), we again emphasize that there is a different indicator of a conditional event $E|H$ for each coherent evaluation of $P(E|H)$.
The same comment applies to the conjunction $(E_1|H_1) \wedge (E_2|H_2)$; indeed, 
each different conjunction  is associated to a different  coherent assessment $(x_1,x_2,x_{12})$. 
We also remark that 	  Definition \ref{CONJUNCTION} is not circular because, after assessing 
 $(x_1,x_2)$,
 the conjunction is completely specified 
 once by the  betting scheme you,  coherently with $(x_1,x_2)$, decide the value $x_{12}=\prev[(E_1H_1E_2H_2+x_{1}\no{H}_1E_2H_2+x_{2}\no{H}_2E_1H_1)|(H_1\vee H_2)]$.
\end{remark}
We  recall a result  which shows that Fr\'echet-Hoeffding bounds still hold for the conjunction of conditional events (\cite[Theorem~7]{GiSa14}).
\begin{theorem}\label{THM:FRECHET}
		Given any coherent assessment $(x_1,x_2)$ on $\{E_1|H_1, E_2|H_2\}$, with $E_1,H_1,E_2$, $H_2$ logically independent, $H_1\neq \emptyset, H_2\neq  \emptyset$, the extension $x_{12} = \mathbb{P}[(E_1|H_1) \wedge (E_2|H_2)]$ is coherent if and only if the following  Fr\'echet-Hoeffding bounds are satisfied:
		\begin{equation}\label{LOW-UPPER}
		\max\{x_1+x_2-1,0\} = x_{12}' \; \leq \; x_{12} \; \leq \; x_{12}'' = \min\{x_1,x_2\} \,.
		\end{equation}
\end{theorem}
\begin{remark}
	From Theorem \ref{THM:FRECHET}, as the assessment  $(x_1,x_2)$ on $\{E_1|H_1, E_2|H_2\}$ is coherent for every $(x_1,x_2)\in[0,1]^2$, the set $\Pi$ of all coherent prevision assessments $(x_1,x_2,x_{12})$ on $\{E_1|H_1, E_2|H_2,(E_1|H_1) \wedge (E_2|H_2)\}$ is 
	\begin{equation}\label{EQ:PI2}
	\Pi=\{(x_1,x_2,x_{12}): (x_1,x_2)\in[0,1]^2, \max\{x_1+x_2-1,0\}  \leq x_{12}\leq \min\{x_1,x_2\}
	\},
	\end{equation}
	which is the tetrahedron with vertices the points $(1,1,1), (1,0,0), (0,1,0), (0,0,0)$.
\end{remark}
Other related approaches to compound conditionals  have been developed in \cite{Kauf09,McGe89}.  However, in our coherence-based approach we can properly manage the case where the probability of some  conditioning events is zero. Then, differently from other authors, we can compute lower and upper bounds for conjunction and disjunction only in terms of the probabilities of the two given conditional events.
We recall below the notion of disjunction between two conditional events.
\begin{definition}\label{DISJUNCTION} Given any pair of conditional events $E_1|H_1$ and $E_2|H_2$, with $P(E_1|H_1)=x_1$ and $P(E_2|H_2)=x_2$, their disjunction
	$(E_1|H_1) \vee (E_2|H_2)$ is  the conditional random quantity  defined as
	\begin{equation}\label{EQ:DISJUNCTION}
	\begin{array}{lll}
	(E_1|H_1) \vee (E_2|H_2)&=&(E_1H_1\vee E_2H_2+x_{1}\no{H}_1\no{E}_2H_2+x_{2}\no{H}_2\no{E}_1H_1)|(H_1\vee H_2)=\\ &=&\left\{\begin{array}{ll}
	1, &\mbox{if $E_1H_1 \vee E_2H_2$ is true,}\\
	0, &\mbox{if $\no{E}_1H_1 \no{E}_2H_2$ is true,}\\
	x_1, &\mbox{if $\no{H}_1\no{E}_2H_2$ is true,}\\
	x_2, &\mbox{if $\no{H}_2\no{E}_1H_1$ is true,}\\
	y_{12}, &\mbox{if $\no{H}_1\no{H}_2$ is true},
	\end{array}
	\right.
	\end{array}
	\end{equation}		
	where $y_{12}=\prev[ (E_1|H_1) \vee (E_2|H_2)]=\prev[(E_1H_1\vee E_2H_2+x_{1}\no{H}_1\no{E}_2H_2+x_{2}\no{H}_2\no{E}_1H_1)|(H_1\vee H_2)]$.
\end{definition}
Of course, the assessment $(x_1,x_2,y_{12})$ must be coherent. In betting terms, $y_{12}$ represents the amount you agree to pay, with the proviso that you will receive the quantity $E_1H_1\vee E_2H_2+x_{1}\no{H}_1\no{E}_2H_2+x_{2}\no{H}_2\no{E}_1H_1+y_{12}\no{H}_1\no{H}_2$, which assumes one of the following values:\vspace{-0.3cm}
\begin{itemize}
	\item $1$, if at least one of the conditional events is true;\vspace{-0.3cm}
	\item  $0$, if both conditional events are false; \vspace{-0.3cm}
	\item  the probability of  the conditional event that is void if one conditional event is void  and the other one is false;\vspace{-0.3cm}
	\item  $y_{12}$ (the amount that you paid) if both 
	conditional events are void.
\end{itemize}
Notice that, differently from conditional events which are three-valued objects, the conjunction $(E_1|H_1) \wedge (E_2|H_2)$ and the  disjunction $(E_1|H_1) \vee (E_2|H_2)$
are not any longer  three-valued objects, but  five-valued objects. Moreover,   the  comments of Remark \ref{REM:CIRC} also apply in a dual way
to disjunction.

We give below the  notion of conjunction  of $n$ conditional events.
\begin{definition}\label{DEF:CONGn}	
	Let  $n$ conditional events $E_1|H_1,\ldots,E_n|H_n$ be given.
	For each  non-empty strict subset $S$  of $\{1,\ldots,n\}$,  let $x_{S}$ be a prevision assessment on $\bigwedge_{i\in S} (E_i|H_i)$.
	Then, the conjunction  $(E_1|H_1) \wedge \cdots \wedge (E_n|H_n)$ is the conditional random quantity $\C_{1\cdots n}$ defined as
	\begin{equation}\label{EQ:CF}
	\begin{array}{lll}
	\C_{1\cdots n}=
	[\bigwedge_{i=1}^n E_iH_i+\sum_{\emptyset \neq S\subset \{1,2\ldots,n\}}x_{S}(\bigwedge_{i\in S} \no{H}_i)\wedge(\bigwedge_{i\notin S} E_i{H}_i)]|(\bigvee_{i=1}^n H_i)=
	\\
	=\left\{
	\begin{array}{llll}
	1, &\mbox{ if } \bigwedge_{i=1}^n E_iH_i\, \mbox{ is true,} \\
	0, &\mbox{ if } \bigvee_{i=1}^n \no{E}_iH_i\, \mbox{ is true}, \\
	x_{S}, &\mbox{ if } (\bigwedge_{i\in S} \no{H}_i)\wedge(\bigwedge_{i\notin S} E_i{H}_i)\, \mbox{ is true}, \; \emptyset \neq S\subset \{1,2\ldots,n\},\\
	x_{1\cdots n}, &\mbox{ if } \bigwedge_{i=1}^n \no{H}_i \mbox{ is true},
	\end{array}
	\right.
	\end{array}
	\end{equation}
	where 
	\[
	x_{1\cdots n}=x_{\{1,\ldots, n\}}=\prev(\C_{1\cdots n})=\prev[(\bigwedge_{i=1}^n E_iH_i+\sum_{\emptyset \neq S\subset \{1,2\ldots,n\}}x_{S}(\bigwedge_{i\in S} \no{H}_i)\wedge(\bigwedge_{i\notin S} E_i{H}_i))|(\bigvee_{i=1}^n H_i)].
	\]	
\end{definition}
For  $n=1$ we obtain $\C_1=E_1|H_1$.  In  Definition \ref{DEF:CONGn}  each possible value $x_S$ of $\C_{1\cdots n}$,  $\emptyset\neq  S\subset \{1,\ldots,n\}$, is evaluated  when defining (in a previous step) the conjunction $\C_{S}=\bigwedge_{i\in S} (E_i|H_i)$. 
Then, after the conditional prevision $x_{1\cdots n}$ is evaluated, $\C_{1\cdots n}$ is completely specified. Of course, 
we require coherence for  the prevision assessment $(x_{S}, \emptyset\neq  S\subseteq \{1,\ldots,n\})$, so that $\C_{1\cdots n}\in[0,1]$.
In the framework of the betting scheme, $x_{1\cdots n}$ is the amount that you agree to pay with the proviso that you will receive:
 \vspace{-0.3cm}
\begin{itemize}
	\item $1$, if all conditional events are true;\vspace{-0.3cm}
\item	$0$, if at least one of the conditional events is false; \vspace{-0.3cm}
\item the prevision of the conjunction of that conditional events which are void,  otherwise. In particular you receive back $x_{1\cdots n}$ when all  conditional events are void.
\end{itemize}
We observe  that conjunction satisfies the monotonicity property  (\cite[Theorem7]{GiSa19}), that is
\begin{equation}\label{EQ:MONOT}
\C_{1\cdots n+1}\leq \C_{1\cdots n}.
\end{equation}
We recall the following result (\cite[Theorem13]{GiSa19}).
\begin{theorem}\label{THM:TEOREMAAI13}
	Let  $n$ conditional events $E_1|H_1,\ldots,E_{n}|H_{n}$ be given, with $x_i=P(E_i|H_i)$, $i=1,\ldots, n$  and $x_{1\cdots n}=\prev(\C_{1 \cdots n })$. Then:
	$
	\max\{x_1+\cdots+x_{n}-n+1,0\}
	\,\,\leq \,\, x_{1\cdots n} \,\,\leq\,\, \min\{x_1,\ldots,x_n\}.
	$
\end{theorem}
We give below the  notion of disjunction  of $n$ conditional events.
\begin{definition}\label{DEF:DISJn}	
Let  $n$ conditional events $E_1|H_1,\ldots,E_n|H_n$ be given.
For each  non-empty strict subset $S$  of $\{1,\ldots,n\}$,  let $y_{S}$ be a prevision assessment on $\bigvee_{i\in S} (E_i|H_i)$.
Then, the disjunction $(E_1|H_1) \vee \cdots \vee (E_n|H_n)$ is the conditional random quantity $\D_{1\cdots n}$ defined as
\begin{equation}\label{EQ:DF}
\begin{array}{lll}
\D_{1\cdots n}=(\bigvee_{i=1}^n E_iH_i+\sum_{\emptyset \neq S\subset \{1,2\ldots,n\}}y_{S}(\bigwedge_{i\in S} \no{H}_i)\wedge(\bigwedge_{i\notin S} \no{E}_i{H}_i))|(\bigvee_{i=1}^n H_i)=\\
=\left\{
\begin{array}{llll}
1, &\mbox{ if } \bigvee_{i=1}^n E_iH_i\, \mbox{ is true,} \\
0, &\mbox{ if } \bigwedge_{i=1}^n \no{E}_iH_i\, \mbox{ is true}, \\
y_{S}, &\mbox{ if } (\bigwedge_{i\in S} \no{H}_i)\wedge(\bigwedge_{i\notin S} \no{E}_i{H}_i)\, \mbox{ is true},  \; \emptyset \neq S\subset \{1,2\ldots,n\},\\
y_{1\cdots n}, &\mbox{ if } \bigwedge_{i=1}^n \no{H}_i \mbox{ is true},
\end{array}
\right.
\end{array}
\end{equation}
where
\[
y_{1\cdots n}=y_{\{1,\ldots, n\}}=\prev(\D_{1\cdots n})=\prev[(\bigvee_{i=1}^n E_iH_i+\sum_{\emptyset \neq S\subset \{1,2\ldots,n\}}y_{S}(\bigwedge_{i\in S} \no{H}_i)\wedge(\bigwedge_{i\notin S} \no{E}_i{H}_i))|(\bigvee_{i=1}^n H_i)].
\]
\end{definition}
For $n=1$ we obtain  $\D_1=E_1|H_1$. 
In the betting framework, you agree to pay $y_{1\cdots n}$ with the proviso that you will receive: \vspace{-0.3cm}
\begin{itemize}
	\item $1$, if at least one of the conditional events
is true;\vspace{-0.3cm}
\item $0$, if  all conditional events are false; \vspace{-0.3cm}
\item the prevision of the disjunction of that conditional events which are void,  otherwise. In particular you receive back $y_{1\cdots n}$ when all  conditional events are void.
\end{itemize}
As we can see from (\ref{EQ:CF}) and (\ref{EQ:DF}), the conjunction $\C_{1\cdots n}$  and the disjunction $\D_{1\cdots n}$ are (in general)  $(2^n+1)$-valued objects because the number of nonempty subsets $S$, and hence the number of possible values $x_S$,  is $2^n-1$. Of course, it may happen that the some of the possible values of $\C_{1\cdots n}$  and  $\D_{1\cdots n}$ coincide.
\begin{remark}\label{REM:CONGCONG}
Given a finite family  $\mathscr{E}$ of  conditional events,  their conjunction and disjunction are also  denoted by $\C(\mathscr{E})$ and $\D(\mathscr{E})$, respectively. We recall that in \cite{GiSa19}, given  two finite families of conditional events $\mathscr{E}'$ and $\mathscr{E}''$,  the objects $\mathcal \C(\mathscr{E}') \wedge  \C(\mathscr{E}'')$  and 
$\mathcal \D(\mathscr{E}') \vee \D(\mathscr{E}'')$
are defined as $\C(\mathscr{E}'\cup \mathscr{E}'')$ and $\D(\mathscr{E}'\cup \mathscr{E}'')$, respectively. Then, it is easy to verify the commutativity and associativity properties  of conjunction and  disjunction  (\cite[Propositions 1 and 2]{GiSa19}).
We recall below the notion of negation for conjoined and disjoined conditionals.
\end{remark}
\begin{definition}
Given  $n$ conditional events $E_1|H_1,\ldots,E_n|H_n$, 
 the negations for the conjunction $\C_{1\cdots n}$  and the disjunction $\D_{1\cdots n}$ are defined  as $\no{\C}_{1\cdots n}=1-\C_{1\cdots n}$ and
 $\no{\D}_{1\cdots n}=1-\D_{1\cdots n}$, respectively.
\end{definition}
Of course, if $n=1$ we obtain $\no{\C}_1=\no{\D}_1=\no{E_1|H_1}=1-E_1|H_1=\no{E}_1|H_1$.
We observe that conjunction and disjunction satisfy  De Morgan’s Laws (\cite[Theorem 5]{GiSa19}), that is
\begin{equation}\label{EQ:DEMORGAN}
\no{\D}_{1\cdots n}=\C_{\no{1}\cdots \no{n}	} \;\;\;\;(\text{i.e., } \D_{1\cdots n}=\no{\C}_{\no{1}\cdots \no{n}	}),\;\;\;\;  \no{\C}_{1\cdots n}=\D_{\no{1}\cdots \no{n}} \;\;\;\;(\text{i.e., } \C_{1\cdots n}=\no{\D}_{\no{1}\cdots \no{n}	}),
\end{equation}
where $\C_{\no{1}\cdots \no{n}	}=\bigwedge_{i=1}^n \no{E}_i|H_i$ and $\D_{\no{1}\cdots \no{n}	}=\bigvee_{i=1}^n \no{E}_i|H_i$.
As shown in  formula (\ref{EQ:DEMORGAN}),  by exploiting  negation,  disjunction could be equivalently defined as $\D_{1\cdots n}=\no{\C}_{\no{1}\cdots \no{n}	}=1-\C_{\no{1}\cdots \no{n}	}$.
\section{A decomposition formula for  conjunctions}
\label{SEC:DECOMP}

In this section we show that
the conjunction $\mathscr{C}_{1\cdots n}$ of $n$ conditional events can
be represented as the sum of two suitable conjunctions of $n+1$ conditional
events. We first give a preliminary result, which is related to Theorem~\ref{THM:FRECHET}, and a remark.
%
\begin{theorem}\label{THM:RISPREL}
	Let  $n$ conditional events $E_1|H_1,\ldots,E_{k}|H_{k},\ldots,E_{n}|H_{n}$  be given, with $E_1,H_1,\ldots,E_{n},H_{n}$ logically independent, and a coherent prevision assessment  $\M=(x_{ S}:\emptyset \neq S \subseteq\{1,\ldots,n\})$ on the family
	$\mathscr{F}=\{\C_{ S}:\emptyset \neq S \subseteq\{1,\ldots,n\}\}$.
	For every $1 \leq k \leq n-1$ it holds that 
	\[
	\max \{0, x_{1 \cdots k}+x_{k+1 \cdots n}-1\} \leq x_{1 \cdots n}  \leq \min \{x_{1 \cdots k}, x_{k+1 \cdots n}\},
	\]
	where 
	\[
	x_{1 \cdots k} = \prev(\C_{1 \cdots k}),\;\;\; x_{k+1 \cdots n} = \prev(\C_{k+1 \cdots n}),\;\;\; x_{1 \cdots n} = \prev(\C_{1 \cdots n}). 
	\]
\end{theorem}
\begin{proof} We set $\M_3=(x_{1 \cdots k}, x_{k+1 \cdots n}, x_{1 \cdots n})$. Moreover, we observe that 
	\[
	\C_{1 \cdots k} \in \{1,0,x_{S'}; S' \subseteq \{1, \ldots, k\}\},\; \C_{k+1 \cdots n} \in \{1,0,x_{S''}; S'' \subseteq \{k+1, \ldots, n\}\}.
	\]
	The possible values $Q_h$'s of  the random vector $(\C_{1 \cdots k}, \C_{k+1 \cdots n}, \C_{1 \cdots n})$ are given in Table \ref{TAB:TABLEFRECHETk}. 
	\begin{table}[!ht]
		\centering
		\begin{tabular}{l|c|c|c|l}
			$C_h$                                  & $\C_{1\cdots k}$ &    $\C_{_{k+1}\cdots  n}$ & $\C_{1 \cdots n\, }$ & $Q_h$  \\
			\hline
			$ \bigwedge_{i=1}^nE_iH_i$           &  $1$  & $1$ & $1$ &   $(1,1,1)$   \\
			$(\bigwedge_{i=1}^k E_iH_i)(\bigvee_{i=k+1}^n \no{E}_iH_i)$           &  $1$  & 0 &$0$  &   $(1,0,0)$   \\
			$(\bigvee_{i=1}^k \no{E}_iH_i)(\bigwedge_{i=k+1}^n E_iH_i)$           &  $0$  & $1$ &$0$  &   $(0,1,0)$    \\
			$(\bigvee_{i=1}^k \no{E}_iH_i)(\bigvee_{i=k+1}^n \no{E}_iH_i)$           &  $0$  & $0$ &$0$  &   $(0,0,0)$   \\
			$ (\bigwedge_{i\notin S''}E_iH_i)(\bigwedge_{i\in S''}\no{H}_i 
			) $           &  $1$  & $x_{S''}$ &$x_{S''}$  &   $(1,x_{S''},x_{S''})$   \\
			$ (\bigvee_{i=1}^k\no{E}_iH_i)
			(\bigwedge_{i \in \{k+1,\ldots,n\}\setminus S''}E_iH_i)
			(\bigwedge_{i\in S''}\no{H}_i 
			) $           &  $0$  & $x_{S''}$ &$0$  &   $(0,x_{S''},0)$   \\
			$ (\bigwedge_{i\notin S'}E_iH_i)(\bigwedge_{i\in S'}\no{H}_i 
			) $           & $x_{S'}$ &1& $x_{S'}$  &   $(x_{S''},1,x_{S''})$   \\
			$ (\bigvee_{i=k+1}^n\no{E}_iH_i)
			(\bigwedge_{i \in \{1,\ldots,k\}\setminus S'}E_iH_i)
			(\bigwedge_{i\in S'}\no{H}_i 
			) $           & $x_{S'}$ &$0$  & $0$ &  $(x_{S'},0,0)$   \\
			$
			(\bigwedge_{i\, \notin\, S' \cup S''}E_iH_i)
			(\bigwedge_{i\,\in\, S' \cup S''}\no{H_i}) 
			$ &   $x_{S'}$ & $x_{S''}$ & $x_{S'\cup S''}$ &  $(x_{S'},x_{S''},x_{S'\cup S''})$\\
			$\bigwedge_{i=1}^n\no{H_i} $ &   $x_{1\cdots k}$ & $x_{k+1\cdots n}$ & $x_{1\cdots n}$ &  $(x_{1 \cdots k},x_{k+1 \cdots n},x_{1 \cdots n})$\\
		\end{tabular}
		\caption{Possible values $Q_h$'s of the random vector $(\C_{1\cdots k},\C_{k+1\cdots  n},\C_{1 \cdots n })$, where $\emptyset\neq S'\subseteq\{1,\ldots,k\}$, ${\emptyset\neq S''\subseteq\{k+1,\ldots,n\}}$,  $S'\cup S''\neq \{1,\ldots,n\}$, and  $(x_{1 \cdots k},x_{k+1 \cdots n},x_{1 \cdots n})=Q_0=\M_3$. } 
		\label{TAB:TABLEFRECHETk}
	\end{table}
	We denote by $\mathcal{T}$ the tetrahedron with vertices $(1,1,1),\, (1,0,0),\, (0,1,0),\, (0,0,0)$, that is 
	\[
	\mathcal{T}=\{(x,y,z): (x,y)\in[0,1]^2, \max \{0, x+y-1\} \leq z  \leq \min \{x, y\}\}.
	\]
	We observe that $\mathcal{T}$ is the convex hull of $(1,1,1),\, (1,0,0),\, (0,1,0),\, (0,0,0)$.
	We also observe that the points 
	$(1, x_{S''}, x_{S''}),\, (0, x_{S''}, 0),\, (x_{S'}, 1, x_{S'}),\, (x_{S'}, 0, 0)$ belong to $\mathcal{T}$ because 
	\[
	\begin{array}{ll}
	(1, x_{S''}, 
	x_{S''})=x_{S''}(1,1,1)+(1-x_{S''})(1,0,0), & (0,x_{S''},0)=x_{S''}(0,1,0)+(1-x_{S''})(0,0,0),\\
	(x_{S'},1, 
	x_{S'})=x_{S'}(1,1,1)+(1-x_{S'})(0,1,0),& ( x_{S'},0, 
	0)=x_{S'}(1,0,0)+(1-x_{S'})(0,0,0).
	\end{array}
	\]
	We recall that coherence of  $\M$ implies coherence of the sub-assessment
	$(x_{i},x_{j},x_{ij})$,   with $i\neq j$, on the sub-family $\{E_i|H_i,E_j|H_j,\mathscr{C}_{ij}\}$.
	By formula (\ref{EQ:PI2}), the coherence of $(x_{i},x_{j},x_{ij})$  amounts to the condition
	$(x_{i},x_{j},x_{ij})\in \mathcal{T}$.
	Now, let us assume by induction that the point $(x_{S'},x_{S''},x_{S' \cup S''})$ belongs to $\mathcal{T}$, for every pair of nonempty subsets $S' \subseteq \{1, \ldots, k\},\; S'' \subseteq \{k+1, \ldots, n\}$, with $S'\cup S''\subset \{1,\ldots,n\}$. Under this inductive hypothesis,  the convex hull of the points $Q_h$'s, with $Q_h \neq Q_0$, is the tetrahedron $\mathcal{T}$. Coherence of $\M_3$ requires that $\M_3$ belongs to the convex hulls of all the points $Q_h$'s ($h\neq 0$), that is  $\M_3 \in \mathcal{T}$. Then,  the inequalities
	\[
	\max \{0, x_{1 \cdots k}+x_{k+1 \cdots n}-1\} \leq x_{1 \cdots n}  \leq \min \{x_{1 \cdots k}, x_{k+1 \cdots n}\},
	\]
	are satisfied.  
\end{proof}
\begin{remark}\label{REM:CnCn+1}
Given the conjunction $	\C_{1\cdots n}$ of $n$ conditional events and  a further conditional event $E_{n+1}|H_{n+1}$, it holds that 
\begin{equation}\label{EQ:CnCn+1}
\begin{array}{lll}
\C_{1\cdots n+1}=\C_{1\cdots n}\wedge (E_{n+1}|H_{n+1})=
\left\{
\begin{array}{llll}
\C_{1\cdots n}, &\mbox{ if } E_{n+1}H_{n+1}\, \mbox{ is true,} \\
0, &\mbox{ if } \no{E}_{n+1}H_{n+1}\, \mbox{ is true}, \\
x_{{S}\cup{\{n+1\}}}, &\mbox{ if }  (\bigwedge_{i\in S} \no{H}_i)\wedge (\bigwedge_{i\notin S} E_i{H}_i) \wedge\no{H}_{n+1}\, \mbox{ is true}.
\end{array}
\right.
\end{array}
\end{equation}	
In particular
\begin{equation}\label{EQ:CONGPART}
\C_{1\cdots n}\wedge 0=0,\;\;\;
\C_{1\cdots n}\wedge 1=\C_{1\cdots n}.
\end{equation}	
Indeed, if $E_{n+1}H_{n+1}=\emptyset$, it holds that $P(E_{n+1}|H_{n+1})=x_{n+1}=0$ and hence $E_{n+1}|H_{n+1}=E_{n+1}H_{n+1}+x_{n+1}\no{H}_{n+1}=0$. As, by (\ref{EQ:MONOT}),  $\C_{1\cdots n+1}\leq E_{n+1}|H_{n+1}=0$, it follows that $\C_{1\cdots n+1}=\C_{1\cdots n}\wedge 0=0$.
	
If $H_{n+1}\subseteq E_{n+1}$, i.e., $E_{n+1}H_{n+1}=H_{n+1}$, it holds that $x_{n+1}=1$ and hence $E_{n+1}|H_{n+1}=E_{n+1}H_{n+1}+x_{n+1}\no{H}_{n+1}=H_{n+1}+\no{H}_{n+1}=1$;	
then   (\ref{EQ:CnCn+1}) becomes 
		\begin{equation*}
	\begin{array}{lll}
	\C_{1\cdots n+1}=
	\left\{
	\begin{array}{llll}
	\C_{1\cdots n}, &\mbox{ if } H_{n+1}\, \mbox{ is true}, \\
	x_{{S}\cup{\{n+1\}}}, &\mbox{ if }  (\bigwedge_{i\in S} \no{H}_i)\wedge (\bigwedge_{i\notin S} E_i{H}_i) \wedge\no{H}_{n+1}\, \mbox{ is true}.
	\end{array}
	\right.
	\end{array}
\end{equation*}	
For every nonempty subset $S\subseteq\{1,\ldots,n\}$, 
by Theorem \ref{THM:RISPREL} it holds that 
\[
\max \{0, x_{S}+x_{n+1}-1\} =x_{S}\leq x_{S\cup\{n+1\}}  \leq x_{S}=\min \{x_{S}, x_{n+1}\}.
\]
Then, $x_{S\cup\{n+1\}}=x_{S}$, which is the value of $\C_{1\cdots n}$ when $(\bigwedge_{i\in S} \no{H}_i)\wedge (\bigwedge_{i\notin S} E_i{H}_i)$ is true. Thus $\C_{1\cdots n+1}=\C_{1\cdots n}\wedge 1= \C_{1\cdots n}$.
\end{remark}
Concerning the decomposition formula, 
 we first examine the case $n=1$ (see also  \cite[Proposition 1]{SaPG17}).  We recall that, given a conditional event   $\C_{1}=E_1|H_1$, we denote its indicator by the same symbol. Then, given a further conditional event $E_2|H_2$, we show that (the indicator) $\C_1$ can be decomposed as the sum of the conjunctions $\C_{12}=(E_1|H_1)\wedge (E_2|H_2)$ and 
 $\C_{1\no{2}}=(E_1|H_1)\wedge (\no{E}_2|H_2)$.
 We set $P(E_1|H_1)=x_1$, $P(E_2|H_2)=x_2$, $\prev(\C_{12})=x_{12},  \prev(\C_{1 \no{2}})=x_{1 \no{2}}$.
The next result shows the decomposition of $\C_{1}$.
\begin{theorem}\label{THM:ADDITIVEn=1}
The conditionals $\C_{1},\C_{12},\C_{1\no{2}}$ satisfy the relation
\begin{equation}\label{EQ:ADDITIVEn=1}
\C_{1}=\C_{12} + \C_{1\no{2}}.
\end{equation}
\end{theorem}
\begin{proof} 
Table~\ref{TAB:TABLEn=1} shows,  under logical independence of the events $E_1,E_2,H_1,H_2$, the possible values for the random vector $(\C_{1},\C_{1\,2},\C_{1\,\no{2}},  \C_{1\,2}+\C_{1\,\no{2}} )$ associated with the constituents $C_h'$s generated by the family $\{\C_{1},\C_{2}\}$.
\begin{table}[!h]
	\centering
	\begin{tabular}{l|l|c|c|c|c}
		& $C_h$                                   &  $\C_{1}$   &  $\C_{1\,2}$ &    $\C_{1\,\no{2}}$ &  $\C_{1\,2}+\C_{1\,\no{2}} $  \\
		\hline
		$C_1$    &$ E_1H_1       E_2H_2 $           &   1 & 1 & 0 &   1   \\
		$C_2$    & $E_1H_1       \no{E}_2H_2$     &   1 & 0 & 1&   1 \\
		$C_{3}$ & $\no{E}_1H_1  E_2H_2 $         &   0 & 0& 0&   0\\ 		
		$C_{4}$ & $\no{E}_1H_1  \no{E}_2H_2$     &  0 & 0 & 0&    0 \\
		$C_{5}$ & $\no{H}_1     E_2H_2$        & $x_1$ & $x_1$& 0&   $x_1$ \\
		$C_{6}$ & $\no{H}_1     \no{E}_2H_2$      & $x_1$ &0 & $x_1$ & $ x_1$\\
		$C_{7}$ & $E_1H_1  \no{H}_2$        &   1 & $x_2$ & $1-x_2$& 1\\ 		
		$C_{8}$ & $\no{E}_1H_1  \no{H}_2$      &   0 & 0 & 0& 0\\
		$ C_0$    &$ \no{H}_1     \no{H}_2$       & $x_1$ &$x_{1\,2}$ & $x_{1\,\no{2}}$ &  $x_{1\,2}+x_{1\,\no{2}}$ \\
		\hline 
	\end{tabular}
	\caption{Numerical values of the random vector $(\C_{1},\C_{1\,2},\C_{1\,\no{2}},  \C_{1\,2}+\C_{1\,\no{2}} )$.}
		\label{TAB:TABLEn=1}
\end{table}
We observe that  both  $C_{1\,2}$ and $C_{1\,\no{2}}$ are conditional random quantities with   the same conditioning event $H_1\vee H_2$ and hence $\C_{1\,2}+\C_{1\,\no{2}}$ is still a conditional random quantity with conditioning event $H_1\vee H_2$. 
As shown in Table \ref{TAB:TABLEn=1}, for each  $C_h\subseteq H_1 \vee H_2$ (i.e., $h=1,\ldots,8$), if $C_h$ is true then  $\C_{1}$ coincides with  $\C_{12}+\C_{1\no{2}}$. In other words, $\C_{1}$ coincides with  $\C_{12}+\C_{1\no{2}}$ when $H_1\vee H_2$ is true.  Thus,  by  Theorem  \ref{THM:EQ-CRQ}, it holds that   
\[
x_1=\prev(\C_{1})=\prev(\C_{12}+\C_{1\no{2}})=\prev(\C_{12})+\prev(\C_{1\no{2}})=x_{12}+x_{1\no{2}};
\]
then   $\C_{1}$ coincides with  $\C_{12}+\C_{1\no{2}}$  when $C_0$ is true. Therefore $\C_{1}$ and $\C_{12}+\C_{1\no{2}}$  coincide in all cases;
that is
$\C_{1}= \C_{12}+\C_{1\no{2}}$.  In case of some logical dependencies, some constituent $C_h$ may be impossible; but,  of course, the relation $\C_{1}= \C_{12}+\C_{1\no{2}}$ is still valid.
\end{proof}
By the same reasoning, $\C_{\no{1}}=\C_{\no{1}2}+\C_{\no{1}\, \no{2}}$, where $\C_{\no{1} 2}=\C_{\no{1}} \wedge \C_{2}$, $\C_{\no{1}\, \no{2}}=\C_{\no{1}} \wedge \C_{\no{2}}$.

 We observe that by Remark \ref{REM:CONGCONG}, given $n+1$ conditional events  $E_1|H_1,\ldots, E_{n+1}|H_{n+1}$, their conjunction  $\C_{1\cdots n+1}$ coincides with $\C_{1\cdots n}\wedge(E_{n+1}|H_{n+1})$. Likewise, $\C_{1\cdots n\no{n+1}}$ coincides with $\C_{1\cdots n}\wedge(\no{E}_{n+1}|H_{n+1})$.
The next result shows the decomposition for the conjunction of $n$ conditional events.
\begin{theorem}\label{THM:ADDITIVEn}
Let  $n+1$ conditional events $E_1|H_1,\ldots, E_{n+1}|H_{n+1}$  be given.   
It holds that 
\begin{equation}\label{EQ:ADDITIVEn}
\C_{1\cdots n}=\C_{1\cdots n+1}+\C_{1\cdots n\, \no{n+1}}.
\end{equation}
\end{theorem}
\begin{proof}
We recall that $x_{1\cdots n}=\prev(\C_{1\cdots n})$, $x_{n+1}=P(E_{n+1}|H_{n+1})$,  $x_{1\cdots n+1}=\prev(\C_{1\cdots n+1})$, and $x_{1\cdots n\, \no{n+1}}=\prev(\C_{1\cdots n\,\no{n+1}})$. Moreover, given any nonempty strict subset $S=\{i_1,\ldots,i_k\}$ of $\{1,2,\ldots,n\}$, we set \[
\C_{S}=\C_{i_1\cdots i_k}=\bigwedge_{j\in S} (E_{j}|H_{j}),\;\; \C_{S \cup \{n+1\}}=\C_{S} \wedge E_{n+1}|H_{n+1},\;\; \C_{S \cup \{\no{n+1}\}}=\C_{	S} \wedge \no{E}_{n+1}|H_{n+1}, 
\]
and 
\[
x_{S}=\prev(\C_{S}),\;\; x_{S \cup \{n+1\}}=\prev(\C_{S \cup \{n+1\}}),\;\; x_{S \cup \{\no{n+1}\}}=\prev(\C_{S \cup \{\no{n+1}\}}). 
\]
We prove the theorem by induction on the cardinality of $S$, denoted by $s$. By Theorem \ref{THM:ADDITIVEn=1}  the equality (\ref{EQ:ADDITIVEn}) holds for $n=1$. We assume that (\ref{EQ:ADDITIVEn}) holds for each integer $s < n$, that is: $\C_{S}=\C_{S \cup \{\no{n+1}\}}+\C_{S \cup \{\no{n+1}\}}$; then, we prove that 
(\ref{EQ:ADDITIVEn}) holds for $s = n$, that is: $\C_{1\cdots n}=\C_{1 \cdots n+1}+\C_{1 \cdots n \no{n+1}}$.\\ We first assume logical independence of the events $E_i,H_i, i=1,\ldots,n+1$. We distinguish the following cases: $(i)$ $E_{n+1}H_{n+1}$ true; $(ii)$ $\no{E}_{n+1}H_{n+1}$ true; $(iii)$ $\no{H}_{n+1}$ true. \\
Case $(i)$.  From (\ref{EQ:CnCn+1})  it holds that $\C_{1 \cdots n+1}  = \C_{1 \cdots n}$ and $\C_{1 \cdots n\,\no{n+1}}= 0$, so that $\C_{1\cdots n}=\C_{1 \cdots n+1}+\C_{1 \cdots n\,\no{n+1}}$. \\
Case $(ii)$. From (\ref{EQ:CnCn+1}) it holds that $\C_{1 \cdots n+1}=0$ and $\C_{1 \cdots n\,\no{n+1}} = \C_{1\cdots n}$, so that $\C_{1\cdots n}=\C_{1 \cdots n+1}+\C_{1 \cdots n\,\no{n+1}}$. \\
Case $(iii)$. We distinguish the following subcases: $(a) \bigwedge_{i=1}^nE_iH_i$ true;
$(b) \bigvee_{i=1}^n\no{E}_iH_i$ true;
$(c)$ $ (\bigwedge_{i\in S} \no{H}_i)\wedge(\bigwedge_{i\notin S} E_i{H}_i)$ true, for some nonempty $S\subset\{1,\ldots,n\}$; $(d)$ $\bigwedge_{i=1}^{n+1} \no{H}_i$ true. \\
In the subcase $(a)$ it holds that $\C_{1 \cdots n}=1$, $\C_{1 \cdots n+1}=x_{n+1}$, and $\C_{1 \cdots n\, \no{n+1}}=1-x_{n+1}$; so that $\C_{1\cdots n}=\C_{1 \cdots n+1}+\C_{1 \cdots n\,\no{n+1}}$.\\
In the subcase $(b)$ it holds that $\C_{1 \cdots n}=\C_{1 \cdots n+1}=\C_{1 \cdots n\, \no{n+1}}=0$; so that $\C_{1\cdots n}=\C_{1 \cdots n+1}+\C_{1 \cdots n\,\no{n+1}}$.\\
In the subcase $(c)$ it holds that $\C_{1 \cdots n}=x_{S}$, $\C_{1 \cdots n+1}=x_{S\cup\{n+1\}}$, and $\C_{1 \cdots n\, \no{n+1}}=x_{S\cup\{\no{n+1}\}}$. By the inductive hypothesis it follows that
 $x_{S}=x_{S\cup\{n+1\}}+x_{S\cup\{\no{n+1}\}}$, 
so that $\C_{1\cdots n}=\C_{1 \cdots n+1}+\C_{1 \cdots n\,\no{n+1}}$.\\
In the subcase $(d)$ it holds that $\C_{1 \cdots n}=x_{1\cdots n}$, $\C_{1 \cdots n+1}=x_{1\cdots n+1}$, and $\C_{1 \cdots n\, \no{n+1}}=x_{1\cdots n\,\no{n+1}}$. 
We observe that $\C_{1 \cdots n}$ is a conditional random quantity with conditioning event  $H_1\vee \cdots \vee H_{n}$. Moreover,  
  both  $\C_{1 \cdots n+1}$ and $\C_{1 \cdots n\, \no{n+1}}$ are conditional random quantities with   the same conditioning event $H_1\vee \cdots \vee H_{n+1}$ and hence $\C_{1 \cdots n+1}+\C_{1 \cdots n\, \no{n+1}}$ is still a conditional random quantity with conditioning event $H_1\vee \cdots \vee H_{n+1}$.  Finally, we observe that 
$\C_{1 \cdots n}$ and $\C_{1 \cdots n+1}+\C_{1 \cdots n\, \no{n+1}}$ coincide
when
 $H_1\vee \cdots \vee H_{n+1}$ is true.  
 Then, by applying  Theorem \ref{THM:EQ-CRQ}  with $X|H=\C_{1 \cdots n}$ and  $Y|K=\C_{1 \cdots n+1}+\C_{1 \cdots n\, \no{n+1}}$, it holds that   
$
x_{1\cdots n}=x_{1\cdots n+1}+x_{1\cdots n\, \no{n+1}}$, so that
$\C_{1 \cdots n}=\C_{1 \cdots n+1}+\C_{1 \cdots n\, \no{n+1}}$.\\
In conclusion, $\C_{1 \cdots n}$ and $\C_{1 \cdots n+1}+\C_{1 \cdots n\, \no{n+1}}$ coincide in all cases; that is $\C_{1 \cdots n}=\C_{1 \cdots n+1}+\C_{1 \cdots n\, \no{n+1}}$ (see also Table~\ref{TAB:TABLEn}). In case of some logical dependencies, some constituent $C_h$ may be impossible; but,  of course, the relation $\C_{1 \cdots n}= \C_{1\cdots  n+1}+\C_{1 \cdots n\, \no{n+1}}$ is still valid.
\begin{table}[!h]
	\centering
	\begin{tabular}{l|c|c|c|c}
 $C_h$                                  & $\C_{1\cdots n}$ &    $\C_{1\cdots  n+1}$ & $\C_{1 \cdots n\, \no{n+1}}$ & $\C_{1\cdots  n+1}+\C_{1 \cdots n\, \no{n+1}}$  \\
 \hline
$ E_{n+1}H_{n+1}$           &  $\C_{1\cdots n}$  & $\C_{1\cdots n}$ & 0 &   $\C_{1\cdots n}$   \\
$ \no{E}_{n+1}H_{n+1}$           &  $\C_{1\cdots n}$  & 0 &$\C_{1\cdots n}$  &   $\C_{1\cdots n}$   \\
 $ (\bigwedge_{i=1}^nE_iH_i) \no{H}_{n+1}$           &  $1$  & $x_{n+1}$ &$1-x_{n+1}$  &   $1$   \\
 $ (\bigvee_{i=1}^n\no{E}_iH_i) \no{H}_{n+1}$           &  $0$  & 0 &0  &   $0$   \\
$  (\bigwedge_{i\in S} \no{H}_i\bigwedge_{i\notin S} E_i{H}_i) \no{H}_{n+1}$ & $x_{S}$ & $x_{S\cup\{n+1\}} $
& $x_{S\cup\{\no{n+1}\}}$ & $x_{S}$\\
$  \bigwedge_{i=1}^{n+1} \no{H}_i$ & $x_{1\cdots n}$& $x_{1\cdots n+1}$ & $x_{1\cdots n\, \no{n+1}}$ &
$x_{1\cdots n}$
	\end{tabular}
	\caption{Numerical values of the conditional random quantities $\C_{1\cdots n},\C_{1\cdots  n+1},\C_{1 \cdots n\, \no{n+1}},\C_{1\cdots  n+1}+\C_{1 \cdots n\, \no{n+1}}$.  Each $S$ is a nonempty strict subset of $\{1,\ldots,n\}$.} 
	\label{TAB:TABLEn}
\end{table}
 \end{proof}
	Given any integer $n\geq 1$  and $n$ conditional events $E_1|H_1,\ldots, E_n|H_n$, we set $\C_{1^* \cdots n^*}=\bigwedge_{i=1}^n E_{i}^*|H_{i}$, where for each index $i$ it holds that $i^*\in\{i,\no{i}\}$ and $E_{i}^*=E_i$, or $E_{i}^*=\no{E}_i$, according to whether $i^*=i$, or $i^*=\no{i}$, respectively. In particular $\C_{1^*}=\C_1=E_1|H_1$ when $1^*=1$ and  $\C_{1^*}=\C_{\no{1}}=\no{E}_1|H_1$ when $1^*=\no{1}$.
Moreover, 	given any subset  $\{i_1, \ldots, i_h\} \subseteq \{1,\ldots,n\}$, by defining $\{i_{h+1}, \ldots, i_n\}=\{1,\ldots,n\}\setminus \{i_1, \ldots, i_h\}$, we set 
	\begin{equation}
\label{EQ:COSTPA}
\C_{i_1\cdots i_h \no{i_{h+1}}\cdots \no{i_n}}=
(E_{i_1}|H_{i_1})\wedge 
\cdots
\wedge 
(E_{i_h}|H_{i_h})
\wedge 
(\no{E}_{i_{h+1}}|H_{i_{h+1}})
\wedge 
\cdots
\wedge 
(\no{E}_{i_{n}}|H_{i_{n}}).
\end{equation}
We recall that by definition the value of  $\C_{i_1\cdots i_h \no{i_{h+1}}\cdots
 \no{i_n}}$, when the conditioning events $H_1,\ldots, H_n$  are all false, is its prevision $\prev(\C_{i_1\cdots i_h \no{i_{h+1}}\cdots \no{i_n}})$. We set
\begin{equation}\label{EQ:LOGPROB}
\prev(\C_{i_1\cdots i_h \no{i_{h+1}}\cdots \no{i_n}})=x_{i_1i_2\cdots i_h \no{i_{h+1}} \cdots \no{i_n}},\;\; \forall\; \{i_1, \ldots, i_h\} \subseteq \{1,\ldots,n\}.
\end{equation}
Notice that, as the operation of conjunction is commutative, for each conjunction $\C_{1^* \cdots n^*}$ it holds that 
$\C_{1^* \cdots n^*} = \C_{i_1 \cdots i_h\no{i_{h+1}} \cdots \no{i_n}}$, for a suitable subset $\{i_1, \ldots, i_h\} \subseteq \{1,\ldots,n\}$. Then, given a further conditional event $E_{n+1}|H_{n+1}$, by the same reasoning of Theorem \ref{THM:ADDITIVEn} it holds that
	\begin{equation}\label{EQ:CnStar}
	\C_{1^* \cdots n^*}=\C_{1^*\cdots  n^*n+1}+\C_{1^* \cdots n^*\, \no{n+1}},\;\;\forall (1^*,\ldots, n^*)\in \{1,\no{1}\}\times \cdots \times \{n,\no{n}\},
	\end{equation}
or equivalently
\begin{equation}\label{EQ:CnStarbis}
	\C_{i_1 \cdots i_h\no{i_{h+1}} \cdots \no{i_n}}=\C_{i_1 \cdots i_h\no{i_{h+1}} \cdots \no{i_n} n+1}+\C_{i_1 \cdots i_h\no{i_{h+1}} \cdots \no{i_n}\, \no{n+1}},\;\;\forall \{i_1, \ldots, i_h\}\subseteq \{1,\ldots,n\}.
	\end{equation}
For instance,  it holds that: \;
$\C_{\no{1}2 \no{3}}=\C_{\no{1}2 \no{3}4}+\C_{\no{1}2 \no{3}\,\no{4}}$; \;  
$\C_{\no{1}2 \no{4}}=
\C_{\no{1}2 \no{4}3}+\C_{\no{1}2 \no{4}\,\no{3}}=
\C_{\no{1}2 3\no{4}}+\C_{\no{1}2 \no{3}\,\no{4}}$, and so on. 
\section{The set of conditional constituents}
\label{SEC:CONDCONST}

In this section we show that a notion of ``constituent'', which we call
\emph{conditional constituent}, can be introduced for the case of $n$ conditional
events $E_{1}|H_{1}, \ldots , E_{n}|H_{n}$. We recall that, given $n$ (unconditional) events
$E_{1},\ldots ,E_{n}$ and denoting the set of their constituents by
$\{C_{h}, h=1,\ldots ,m\}$, where $m\leqslant 2^{n}$ (with $m=2^{n}$ in case
of logical independence), it holds that
%
\begin{equation}\label{EQ:PARTIZIONE}
\begin{array}{ll}
(i) \;\; C_h \wedge C_k = \emptyset,\, \forall \, h \neq k; \;\;
& (ii) \;\;\bigvee_{h=1}^{m} C_h  = \Omega. 
\end{array}
\end{equation}
In terms of indicators (denoted by the same symbols) formula (\ref{EQ:PARTIZIONE}) becomes: 
\begin{equation}\label{EQ:PARTIZIONE'}
\begin{array}{ll}
(i)' \;\; C_h \wedge C_k = 0,\, \forall \, h \neq k; \;\;
& (ii)' \;\;\sum_{h=1}^{m} C_h  = 1,
\end{array}
\end{equation}
with $C_h\geq0$, $h=1,\ldots,m$.
Then,  it holds that: 
\begin{equation}\label{EQ:PROPADD}
\begin{array}{ll}
E_j =\sum_{h:\,C_h \subseteq E_j} C_h, &  P(E_j) = \sum_{h:\,C_h \subseteq E_j} P(C_h);\;\; j=1,\ldots,n.
\end{array}
\end{equation}
	We introduce  the set of conditional
constituents associated with $n$ conditional events, by obtaining some
properties which are analogous to those valid for the unconditional events.
Indeed, we will show that properties $(i)'$ and $(ii)'$ in (\ref{EQ:PARTIZIONE'}),
still hold if we replace events, and their constituents, by conditional
events, and their conditional constituents, respectively. In other words,
the conditional constituents are \emph{incompatible} (i.e., their conjunction
is 0) and their sum is 1. 

Moreover, likewise formula (\ref{EQ:PROPADD}), we will show that the indicator
of each conditional event, and its prevision, can be decomposed as the
sum of suitable conditional constituents, and their previsions, respectively.

In addition, as in the case of unconditional events, the conditional constituents
associated with a family of $n$ conditional events
$\{E_{1}|H_{1},\ldots , E_{n}|H_{n}\}$ are all the (non zero) conjunctions
$(A_{1}|H_{1})\wedge \cdots \wedge (A_{n}|H_{n})$, where
$A_{i}\in \{E_{i},\overline{E}_{i}\}$, $i=1,\ldots ,n$.
%
%
%
\begin{definition}
 The set of conditional constituents, or \emph{c-constituents}, associated with  a family  of $n$ conditional events $\mathscr{E}=\{E_1|H_1,\ldots, E_n|H_n\}$ is
\[
\K=\{\C_{i_1\cdots i_h \no{i_{h+1}} \cdots \no{i_n}}:\;\; \{i_1,\cdots, i_h\}\subseteq \{1,\ldots,n\}, \C_{i_1\cdots i_h \no{i_{h+1}} \cdots \no{i_n}}\neq 0\},
\]
where each  c-constituent $\C_{i_1\cdots i_h \no{i_{h+1}}\cdots \no{i_n}}$  is a conjunction, as defined in $(\ref{EQ:COSTPA})$.
\end{definition}	
Notice that the cardinality of $\K$ is $2^n$ when 
the events $E_1,\ldots,E_n, H_1,\ldots, H_n$ are  logically independent.
In the presence of some logical dependencies it may be that $\C_{i_1i_2\cdots i_h \no{i_{h+1}} \cdots \no{i_n}}=0$ for some $\{i_1, \ldots, i_h\}\subseteq \{1,\ldots,n\}$, as shown in the example below. If $\C_{i_1i_2\cdots i_h \no{i_{h+1}} \cdots \no{i_n}}$ coincides with 0, then it is not included in the set $\K$.
\begin{example}
Given two logically independent events $E,H$ let us consider the  family $\mathscr{E}=\{E_1|H_1,E_2|H_2\}$, where $E_1=E, E_2=\no{E}, H_1=H_2=H$. 
We observe that $\K \subseteq \{\C_{12},\C_{1\no{2}},\C_{\no{1}2},\C_{\no{1}\,\no{2}}\}$, where, by recalling  that $E|H=EH+P(E|H)\no{H}$ and hence $\emptyset|H=0$, it holds that 
\[
\C_{12}=(E|H)\wedge (\no{E}|H)=\emptyset|H=0=\C_{\no{1}\,\no{2}},\, \C_{1\no{2}}=(E|H)\wedge (E|H)=E|H,\,
\C_{\no{1}2}=(\no{E}|H)\wedge (\no{E}|H)=\no{E}|H.
\]
As we can see, in this case there are two c-constituents which are not zero; that is: $\K = \{\C_{1\no{2}},\C_{\no{1}2}\} = \{E|H,\no{E}|H\} = \mathscr{E}$.
\end{example}
In the next result we show  that the properties $(i)'$ and $(ii)'$ in (\ref{EQ:PARTIZIONE'}), 
relative to  unconditional events, still hold for the case of  conditional events. 
\begin{theorem}\label{THM:INCOMP}
	Given a family of $n$ conditional events $\mathscr{E}=\{E_1|H_1,\ldots, E_{n}|H_{n}\}$, let $\K$ be the set of c-constituents associated with $\mathscr{E}$. 	It holds that 
	\begin{equation}\label{EQ:INCOMPCONG}
	\C_{i_1\cdots i_h \no{i_{h+1}}\cdots \no{i_n}}
	\wedge \C_{j_1\cdots j_k \no{j_{k+1}}\cdots \no{j_n}}=\emptyset|(H_1\vee \cdots \vee H_n)=0,  \;\; \forall \{i_1,i_2,\ldots,i_h\}\neq  \{j_1,j_2,\ldots,j_k\}.
	\end{equation}
\end{theorem}
\begin{proof}
As $\{i_1,\ldots,i_h\}\neq  \{j_1,\ldots,j_k\}$, the set
$(\{i_1,\ldots,i_h\} \setminus  \{j_1,\ldots,j_k\})\cup(  \{j_1,\ldots,j_k\} \setminus \{i_1,\ldots,i_h\}) 
$ is non empty. Let 
$r$ be one of its elements.
For the sake of simplicity, we assume  that $r=i_1=j_{k+1}$.  Then
\[
(E_{i_1}|H_{i_1})\wedge (\no{E}_{j_{k+1}}|H_{j_{k+1}})=
(E_{i_1}|H_{i_1})\wedge (\no{E}_{i_1}|H_{i_1})=(E_{i_1}\wedge \no{E}_{i_1})|H_{i_1}=\emptyset|H_{i_1}=0,
\]
and hence
\[
\begin{array}{ll}
\C_{i_1\cdots i_h \no{i_{h+1}}\cdots \no{i_n}}
\wedge \C_{j_1\cdots j_k \no{j_{k+1}}\cdots \no{j_n}}=(E_{i_1}|H_{i_1})\wedge (\no{E}_{j_{k+1}}|H_{j_{k+1}})
\wedge \C_{i_2\cdots i_h \no{i_{h+1}}\cdots \no{i_n}}
\wedge \C_{j_1\cdots j_k \no{j_{k+2}}\cdots \no{j_n}}=\\
=(\emptyset|H_{i_1})  \wedge \C_{i_2\cdots i_h \no{i_{h+1}}\cdots \no{i_n}}
\wedge \C_{j_1\cdots j_k \no{j_{k+2}}\cdots \no{j_n}}=\emptyset|(H_1 \vee \cdots \vee H_n)=0.
\end{array}
\]
\end{proof}
\begin{theorem}\label{THM:DECOMP}
		Given a family of $n$ conditional events $\mathscr{E}=\{E_1|H_1,\ldots, E_{n}|H_{n}\}$, let $\K$ be the set of c-constituents associated with $\mathscr{E}$. For each $1\leq k\leq n$, it holds that \begin{equation}\label{EQ:DECOMPGEN}
\begin{array}{ll}
\sum_{\{i_1,\ldots,i_h\}\subseteq\{1,\ldots,k\}}\C_{i_1\cdots i_h \no{i_{h+1}}\cdots \no{i_k}}=
\sum_{(1^*,\ldots,k^*)\in\{1,\no{1}\}\times \cdots \times \{k,\no{k}\}}\C_{1^*\cdots k^*}=1,\\ \\
\C_{i_1\cdots i_h \no{i_{h+1}}\cdots \no{i_k}}\geq 0, \;\; \forall\;\{i_1,\ldots,i_h\}\subseteq\{1,\ldots,k\}. 
\end{array}
\end{equation} 
\end{theorem}
\begin{proof}
	First of all, as $P(\Omega|H_1)=1$, we observe that \[\C_1+\C_{\no{1}}=E_1|H_1+\no{E}_1|H_1=(E_1\vee \no{E}_1)|H_1=\Omega|H_1=1,\]
	that is  (\ref{EQ:DECOMPGEN}) holds when $k=1$. Moreover,  from  	 
		(\ref{EQ:ADDITIVEn=1}), $\C_{12}+
		\C_{1\no{2}}+\C_{\no{1}2}+
		\C_{\no{1}\,\no{2}}=\C_{1}+\C_{\no{1}}=1$, that is (\ref{EQ:DECOMPGEN}) holds when $k=2$. Then, by induction, assuming  (\ref{EQ:DECOMPGEN}) valid for $k-1$, from (\ref{EQ:CnStar}) it follows that 
	\[
\begin{array}{ll}	
\sum_{\{i_1,\ldots,i_h\}\subseteq\{1,\ldots,k\}}\C_{i_1\cdots i_h \no{i_{h+1}}\cdots \no{i_k}}=
\sum_{(1^*,\ldots,k^*)\in\{1,\no{1}\}\times \cdots \times \{k,\no{k}\}}\C_{1^*\cdots k^*}=\\
=\sum_{(1^*,\ldots,(k-1)^*)\in\{1,\no{1}\}\times \cdots \times \{k-1,\no{k-1}\}}(\C_{1^*\cdots (k-1)^*k}+
\C_{1^*\cdots (k-1)^*\no{k}})=\\
=\sum_{(1^*,\ldots,(k-1)^*)\in\{1,\no{1}\}\times \cdots \times \{k-1,\no{k-1}\}}\C_{1^*\cdots (k-1)^*}=1,
\end{array}
\]	
that is formula	(\ref{EQ:DECOMPGEN}) is valid for  $k$. Finally, the inequalities $\C_{i_1\cdots i_h \no{i_{h+1}}\cdots \no{i_k}}\geq 0$,  $\forall\;\{i_1,\ldots,i_h\}\subseteq\{1,\ldots,k\}$, hold because the prevision assessments used when defining conjunctions are assumed to be coherent.
\end{proof}
We observe in particular that from Definition \ref{DEF:CONGn} and Theorem \ref{THM:DECOMP}  it holds  that 
\begin{equation} \label{EQ:CONBAS}
\sum_{\{i_1,i_2,\ldots,i_h\}\subseteq\{1,2,\ldots,n\}}x_{i_1\cdots i_h \no{i_{h+1}}\cdots \no{i_n}}=1, \,\,
x_{i_1\cdots i_h \no{i_{h+1}}\cdots \no{i_n}}\geq 0, \,\,\,\, \forall\,  \{i_1,\ldots, i_h\}\subseteq  \{1,\dots,n\}.
\end{equation}
In other words the prevision of each  conditional constituent is nonnegative and the sum of all these previsions is equal to 1.
In addition, we show that 
the properties in   (\ref{EQ:PROPADD}) still  hold for conditional events.
Indeed, 
by Theorem~\ref{THM:ADDITIVEn},  it follows that
\[
\begin{array}{ll}
\C_1=\C_{12}+\C_{1\no{2}}=\C_{123}+\C_{12\no{3}}+\C_{1\no{2}3}+\C_{1\no{2}\,\no{3}}=\ldots=\\
=\C_{12\cdots n}+
\C_{12\cdots n-1 \no{n}}
+
\cdots 
+\C_{1 \no{2}\cdots \no{n-1}n}
+\C_{1 \no{2}\cdots \no{n}}
=\\
=
\sum_{(2^*,\ldots,n^*)\in\{2,\no{2}\}\times \cdots \times \{n,\no{n}\}}\C_{12^*\cdots n^*}
=\sum_{ \{1\}\subseteq \{i_1,\ldots,i_h\}\subseteq \{1,\ldots,n\}}
\C_{i_1\cdots i_h \no{i_{h+1}}\cdots \no{i_n}},
\end{array}
\]	
and in general, for each  $j\in\{1,\ldots,n\}$ it holds that
 \begin{equation}\label{EQ:DECOMPPART}
\begin{array}{ll}
\displaystyle
\C_j=\sum_{\{(1^*,\ldots, (j-1)^*,(j+1)^*,\ldots,n^*)\}}\C_{1^*\cdots (j-1)^*j(j+1)^*\cdots n^*}
\displaystyle=\sum_{ \{j\}\subseteq \{i_1,\ldots,i_h\}\subseteq \{1,\ldots,n\}}
\C_{i_1i_2\cdots i_h \no{i_{h+1}}\cdots \no{i_n}},
\end{array}
\end{equation}
where the symbol
${\{(1^*,\ldots, (j-1)^*,(j+1)^*,\ldots,n^*)\}}$ denotes the following set
\[
\{(1^*,\ldots, (j-1)^*,(j+1)^*,\ldots,n^*)\in\{1,\no{1}\}\times \cdots\times \{j-1,\no{j-1}\} \times \{j+1,\no{j+1}\}   \times \cdots \times \{n,\no{n}\}\}.
\]
Moreover, concerning the probability  $x_j$ of $E_j|H_j$, from (\ref{EQ:DECOMPPART}) it holds that
\begin{equation}\label{EQ:DECOMPPARTPROB}
\begin{array}{ll}
\displaystyle
P(\C_j)=x_j=\sum_{\{(1^*,\ldots, (j-1)^*,(j+1)^*,\ldots,n^*)\}}x_{1^*\cdots (j-1)^*j(j+1)^*\cdots n^*}
\displaystyle=\\=
\displaystyle\sum_{ \{j\}\subseteq \{i_1,\ldots,i_h\}\subseteq \{1,\ldots,n\}}
x_{i_1i_2\cdots i_h \no{i_{h+1}}\cdots \no{i_n}}=
\displaystyle\sum_{ \{j\}\subseteq \{i_1,\ldots,i_h\}\subseteq \{1,\ldots,n\}}
\prev(\mathscr{C}_{i_1i_2\cdots i_h \no{i_{h+1}}\cdots \no{i_n}}).
\end{array}
\end{equation}
More in general, for the conjunction $\C_{S}=\bigwedge_{j\in S} (E_{j}|H_{j})$ it holds that
\begin{equation}\label{}
\C_{S}=\sum_{\{i_1,\ldots,i_h\}\supseteq S }
\C_{{i_1} \cdots  i_h \no{i_{h+1}}\cdots \no{i_n}},
\;\; \forall \emptyset \neq  S\subseteq  \{1,\ldots,n\},
\end{equation}
and hence
\begin{equation}\label{EQ:XS}
\prev(\C_{S})=x_{S}=\sum_{\{i_1,\ldots,i_h\}\supseteq S }
x_{{i_1} \cdots  i_h \no{i_{h+1}}\cdots \no{i_n}},\;\; \forall \emptyset \neq  S\subseteq  \{1,\ldots,n\}.
\;
\end{equation}
Moreover,
 \begin{equation}\label{EQ:DECOMPPARTGEN}
\C_{i_1i_2\cdots i_h \no{i_{h+1}} \cdots \no{i_k}}=
\sum_{\{j_1,\ldots,j_r\}\subseteq \{i_{k+1},\ldots,i_n\}}
\C_{i_1i_2\cdots i_h j_1\cdots j_r  \no{i_{h+1}} \cdots \no{i_k} \no{j_{r+1}} \cdots \no{j_{n-k-r} }}, \,\,\forall\, 0\leq h\leq k\leq n,
\end{equation}
and for its prevision $x_{i_1i_2\cdots i_h \no{i_{h+1}} \cdots \no{i_k}}$ it holds that
\begin{equation}\label{EQ:DECOMPPARTGENPROB}
x_{i_1i_2\cdots i_h \no{i_{h+1}} \cdots \no{i_k}}=
\sum_{\{j_1,\ldots,j_r\}\subseteq \{i_{k+1},\ldots,i_n\}}
x_{i_1i_2\cdots i_h j_1\cdots j_r  \no{i_{h+1}} \cdots \no{i_k} \no{j_{r+1}} \cdots \no{j_{n-k-r} }}, \,\,\forall\, 0\leq h\leq k\leq n.
\end{equation}
\section{The inclusion-exclusion principle and the distributivity property}
\label{SEC:INEX}
In this section we show that the well known inclusion-exclusion formula, which holds for the disjunction of $n$ unconditional events (and its probability), still holds  
 for the disjunction of $n$ conditional events (and its prevision). This result, and other related formulas, will be used in Section \ref{SEC:NECSUFF}. We also prove a distributivity property by means of which we can directly derive the inclusion-exclusion formula.
We first give a preliminary result.
\begin{theorem}
Given  $n+1$ conditional events $E_1|H_1, \ldots, E_{n+1}|H_{n+1}$, it holds that
\begin{equation}\label{EQ:PREINCLU}
\begin{array}{ll}
\C_{\no{1}\cdots \no{h}\, i_1 \cdots  i_k\, n+1}=\C_{i_1 \cdots  i_k\, n+1}-\sum_{j=1}^{h}\C_{j i_1 \cdots  i_k\, n+1}+\sum_{1\leq j_1<j_2\leq h}\C_{j_1j_2 i_1 \cdots  i_k\, n+1}+\cdots+(-1)^{h} \C_{1\cdots h\, i_1 \cdots  i_k\, n+1}, \\
  1\,\leq h\leq n,\, \{i_1,\ldots,i_k\}\subseteq \{h+1,\ldots,n\}.
\end{array}
\end{equation}
\end{theorem}
\begin{proof}
Formula (\ref{EQ:PREINCLU}) is satisfied for $h=1$  because,  by the decomposition formula (\ref{EQ:CnStarbis}), it holds that
\[
\C_{\no{1}\, i_1 \cdots  i_k\, n+1}=\C_{ i_1 \cdots  i_k\, n+1}-\C_{1 i_1 \cdots  i_k\, n+1},\,\; \forall \{i_1,\ldots, i_k\}\subseteq \{2,\ldots,n\}.
\]
By assuming that (\ref{EQ:PREINCLU}) is satisfied for $h\leq n-1$,  we prove that (\ref{EQ:PREINCLU}) is also satisfied for $h+1$.
By (\ref{EQ:CnStarbis}), it holds that
\[
\begin{array}{ll}
\C_{\no{1}\cdots \no{h}\, i_1 \cdots  i_k\, n+1}=\C_{\no{1}\cdots \no{h}h+1\, i_1 \cdots  i_k\, n+1}+\C_{\no{1}\cdots \no{h+1}\, i_1 \cdots  i_k\, n+1},\ \ \  \forall\, \{i_1,\ldots, i_k\}\subseteq \{h+2,\ldots,n\}.
\end{array}
\] Moreover, 
 by the hypothesis,  it holds that  
\[
\C_{\no{1}\cdots \no{h}\, i_1 \cdots  i_k\, n+1}=\C_{i_1 \cdots  i_k\, n+1}-\sum_{j=1}^{h}\C_{j i_1 \cdots  i_k\, n+1}+\sum_{1\leq j_1<j_2\leq h}\C_{j_1j_2 i_1 \cdots  i_k\, n+1}+\cdots+(-1)^{h} \C_{1\cdots h\, i_1 \cdots  i_k\, n+1}
\]
and
\[
\C_{\no{1}\cdots \no{h}\,h+1\, i_1 \cdots  i_k\, n+1}=\C_{h+1\,i_1 \cdots  i_k\, n+1}-\sum_{j=1}^{h}\C_{j h+1i_1 \cdots  i_k\, n+1}+\sum_{1\leq j_1<j_2\leq h}\C_{j_1j_2 h+1 i_1 \cdots  i_k\, n+1}+\cdots+(-1)^{h} \C_{1\cdots h+1 i_1 \cdots  i_k\, n+1}.
\]
Then, by (\ref{EQ:CnStarbis}), for all $\{i_1,\ldots,i_k\}\subseteq \{h+2,\ldots,n\}$ it follows that 
\[
\begin{array}{ll}
\C_{\no{1}\cdots \no{h+1}\, i_1 \cdots  i_k\, n+1}=\C_{\no{1}\cdots \no{h}\, i_1 \cdots  i_k\, n+1}-\C_{\no{1}\cdots \no{h}h+1\, i_1 \cdots  i_k\, n+1}=\\
=[\C_{i_1 \cdots  i_k\, n+1}-\sum_{j=1}^{h}\C_{j i_1 \cdots  i_k\, n+1}+\sum_{1\leq j_1<j_2\leq h}\C_{j_1j_2 i_1 \cdots  i_k\, n+1}+\cdots+(-1)^{h} \C_{1\cdots h\, i_1 \cdots  i_k\, n+1}]+\\
-[\C_{h+1\,i_1 \cdots  i_k\, n+1}-\sum_{j=1}^{h}\C_{j h+1i_1 \cdots  i_k\, n+1}+\sum_{1\leq j_1<j_2\leq h}\C_{j_1j_2 h+1 i_1 \cdots  i_k\, n+1}+\cdots+(-1)^{h} \C_{1\cdots h+1 i_1 \cdots  i_k\, n+1}]=\\
= \C_{i_1 \cdots  i_k\, n+1}-\sum_{j=1}^{h+1}\C_{j i_1 \cdots  i_k\, n+1}+\sum_{1\leq j_1<j_2\leq h+1}\C_{j_1j_2 i_1 \cdots  i_k\, n+1}+\cdots+(-1)^{h+1} \C_{1\cdots h+1\, i_1 \cdots  i_k\, n+1}.
\end{array}
\]
Then, formula (\ref{EQ:PREINCLU}) follows by iterating the previous reasoning from $h=0$ to $h=n-1$.
\end{proof}
We give below some examples where  formula (\ref{EQ:PREINCLU}) is obtained.
\begin{equation*}
\begin{array}{l}
\mathscr{C}_{\overline{1}2}=\mathscr{C}_{2}-\mathscr{C}_{12},\;
\mathscr{C}_{\overline{1}23}=\mathscr{C}_{23}-\mathscr{C}_{123},\;
\mathscr{C}_{\overline{1}\,\overline{2}3}=\mathscr{C}_{\overline{1}3}-
\mathscr{C}_{\overline{1}23}=\mathscr{C}_{3}-\mathscr{C}_{13}-
\mathscr{C}_{23}+\mathscr{C}_{123},\; \\ \mathscr{C}_{\overline{1}\,
	\overline{2}34}=\mathscr{C}_{\overline{1}34}-\mathscr{C}_{
	\overline{1}234}=\mathscr{C}_{34}-\mathscr{C}_{134}-\mathscr{C}_{234}+
\mathscr{C}_{1234},
\\
\mathscr{C}_{\overline{1}\,\overline{2}\,\overline{3}4}=\mathscr{C}_{
	\overline{1}\,\overline{2}4}-\mathscr{C}_{\overline{1}\,\overline{2}34}=
\mathscr{C}_{\overline{1}4}-\mathscr{C}_{\overline{1}24}-
\mathscr{C}_{\overline{1}34}+\mathscr{C}_{\overline{1}234}=
\mathscr{C}_{4}-\mathscr{C}_{14}-\mathscr{C}_{24}+\mathscr{C}_{124}-
\mathscr{C}_{34}+\mathscr{C}_{134}+\mathscr{C}_{234}-\mathscr{C}_{1234}.
\end{array}
\end{equation*}
In the next result  we obtain the inclusion-exclusion formula for the disjunction of $n$ conditional events.
\begin{theorem}\label{THM:INCLEXCLNEW}
Given $n$ conditional events $E_1|H_1, \ldots, E_n|H_n$, it holds that
\[
\D_{1\cdots n}=\sum_{h=1}^{n}(-1)^{h+1} \sum_{1\leq i_1<\cdots <i_h\leq n}\C_{i_1\cdots i_h}=\sum_{i=1}^n\C_{i}-\sum_{1\leq i_1<i_2\leq n}\C_{i_1i_2}+\cdots+(-1)^{n+1} \C_{1\cdots n}.
\]
\end{theorem}
\begin{proof} 
From  (\ref{EQ:CONGPART}), it holds  that $\C_{\no{1}\cdots \no{n}}\wedge 1=
\C_{\no{1}\cdots \no{n}}.$
Then, by applying (\ref{EQ:PREINCLU}) with $h=n$ and $\C_{n+1}=1$,
it follows that
\begin{equation}\label{EQ:INCLESCLCONG}
\C_{\no{1}\cdots \no{n}}=
1-\sum_{i=1}^{n}\C_{i}+\sum_{1\leq i_1<i_2\leq n}\C_{i_1i_2}+\cdots+(-1)^{n} \C_{1\cdots n}.
\end{equation}
Finally, by recalling (\ref{EQ:DEMORGAN}), we obtain
\[
\D_{1\cdots n}=1-\C_{\no{1} \cdots \no{n}}=\sum_{i=1}^n\C_{i}-\sum_{1\leq i_1<i_2\leq n}\C_{i_1i_2}+\cdots+(-1)^{n+1} \C_{1\cdots n}.
\]
\end{proof}
In the next result we prove the validity of a  suitable distributivity property.
\begin{theorem}\label{THM:EQDISTRIB}
	Let   $\C_1,\ldots,\C_{n+1}$ be  $n+1$ conditional events.  Then, the following distributivity property is satisfied:
\begin{equation}\label{EQ:DISTRIBGENERAL}
	\begin{array}{ll}
	\displaystyle
	[1-\sum_{i=1}^{h}\C_{i}+\sum_{1\leq i_1<i_2\leq h}\C_{i_1i_2}+\cdots+(-1)^{h} \C_{1\cdots h}]\wedge \C_{ i_1 \cdots  i_k\, n+1}= \\ 
	\displaystyle
	= 1\wedge \C_{ i_1 \cdots  i_k\, n+1}-\sum_{i=1}^{h}\C_{i}\wedge \C_{ i_1 \cdots  i_k\, n+1}+\sum_{1\leq i_1<i_2\leq h}\C_{i_1i_2}\wedge \C_{ i_1 \cdots  i_k\, n+1}+\cdots+(-1)^{h} \C_{1\cdots h}\wedge \C_{ i_1 \cdots  i_k\, n+1},\\
  1\,\leq h\leq n,\, \{i_1,\ldots,i_k\}\subseteq \{h+1,\ldots,n\}.
	\end{array}
\end{equation}
\end{theorem}
\begin{proof}  By  recalling Remark \ref{REM:CONGCONG}, formulas (\ref{EQ:PREINCLU}),  and  (\ref{EQ:INCLESCLCONG}), it holds that  
\[ 
\begin{array}{ll}
\displaystyle
[1-\sum_{i=1}^{h}\C_{i}+\sum_{1\leq i_1<i_2\leq h}\C_{i_1i_2}+\cdots+(-1)^{h} \C_{1\cdots h}]\wedge \C_{ i_1 \cdots  i_k\, n+1}= 
\C_{\no{1}\cdots \no{h}} \wedge \C_{i_1 \cdots  i_k\, n+1} = \C_{\no{1}\cdots \no{h}i_1 \cdots  i_k\, n+1} = \\ 
\displaystyle=\C_{i_1 \cdots  i_k\, n+1}-\sum_{j=1}^{h}\C_{j i_1 \cdots  i_k\, n+1}+\sum_{1\leq j_1<j_2\leq h}\C_{j_1j_2 i_1 \cdots  i_k\, n+1}+\cdots+(-1)^{h} \C_{1\cdots h\, i_1 \cdots  i_k\, n+1}= \\
\displaystyle=1\wedge \C_{ i_1 \cdots  i_k\, n+1}-\sum_{i=1}^{h}\C_{i}\wedge \C_{ i_1 \cdots  i_k\, n+1}+\sum_{1\leq i_1<i_2\leq h}\C_{i_1i_2}\wedge \C_{ i_1 \cdots  i_k\, n+1}+\cdots+(-1)^{h} \C_{1\cdots h}\wedge \C_{ i_1 \cdots  i_k\, n+1}.
\end{array} 
\]
\end{proof}	
The next result shows a further aspect of the distributivity property.
\begin{theorem}\label{THM:DISTRIBEXPLBIS}
Let   $\C_1,\ldots,\C_{n+1}$ be  $n+1$ conditional events.  Then,
\begin{equation}\label{EQ:DISTRIB1-Cn+1} 
\begin{array}{ll}
[1-\sum_{i=1}^{n}\C_{i}+\sum_{1\leq i_1<i_2\leq n}\C_{i_1i_2}+\cdots+(-1)^{n} \C_{1\cdots n}]\wedge (1-\C_{n+1})= \\
=[1-\sum_{i=1}^{n}\C_{i}+\sum_{1\leq i_1<i_2\leq n}\C_{i_1i_2}+\cdots+(-1)^{n} \C_{1\cdots n}]\wedge 1+\\
-[1-\sum_{i=1}^{n}\C_{i}+\sum_{1\leq i_1<i_2\leq n}\C_{i_1i_2}+\cdots+(-1)^{n} \C_{1\cdots n}]\wedge \C_{n+1}.\\
\end{array}
\end{equation}
\end{theorem}
\begin{proof}
We observe that, by  Theorem \ref{THM:EQDISTRIB},
when $h=n$ it holds that
\begin{equation}\label{EQ:DISTRIBCn+1}
	\begin{array}{ll}
[1-\sum_{i=1}^{n}\C_{i}+\sum_{1\leq i_1<i_2\leq n}\C_{i_1i_2}+\cdots+(-1)^{n} \C_{1\cdots n}]\wedge \C_{n+1}= \\ 
	= 1\wedge\C_{n+1}-\sum_{i=1}^{n}\C_{i}\wedge \C_{n+1}+\sum_{1\leq i_1<i_2\leq n}\C_{i_1i_2}\wedge \C_{n+1}+\cdots+(-1)^{n} \C_{1\cdots h}\wedge \C_{n+1}.
	\end{array}
\end{equation}
In particular, if $\C_{n+1}=1$	it follows that
\begin{equation}\label{EQ:DISTRIBCn+1=1}
\begin{array}{ll}
	[1-\sum_{i=1}^{n}\C_{i}+\sum_{1\leq i_1<i_2\leq n}\C_{i_1i_2}+\cdots+(-1)^{n} \C_{1\cdots n}]\wedge 1= \\ 
	= 1\wedge1-\sum_{i=1}^{n}\C_{i}\wedge 1+\sum_{1\leq i_1<i_2\leq n}\C_{i_1i_2}\wedge 1+\cdots+(-1)^{n} \C_{1\cdots h}\wedge 1.
\end{array}
\end{equation}
Based on (\ref{EQ:INCLESCLCONG}), (\ref{EQ:DISTRIBCn+1}), and (\ref{EQ:DISTRIBCn+1=1}) it follows that 
\[ 
\begin{array}{ll}
[1-\sum_{i=1}^{n}\C_{i}+\sum_{1\leq i_1<i_2\leq n}\C_{i_1i_2}+\cdots+(-1)^{n} \C_{1\cdots n}]\wedge (1-\C_{n+1})= 
\C_{\no{1}\cdots \no{n}} \wedge \C_{\no{n+1}} = \\ =\C_{\no{1}\cdots \no{n+1}}=1-\sum_{i=1}^{n+1}\C_{i}+\sum_{1\leq i_1<i_2\leq n+1}\C_{i_1i_2}+\cdots+(-1)^{n+1} \C_{1\cdots n+1}= \\
=[1-\sum_{i=1}^{n}\C_{i}+\sum_{1\leq i_1<i_2\leq n}\C_{i_1i_2}+\cdots+(-1)^{n} \C_{1\cdots n}] + \\ - [\C_{n+1}-\sum_{i=1}^{n}\C_{i\,n+1}+\sum_{1\leq i_1<i_2\leq n}\C_{i_1i_2 n+1}+\cdots+(-1)^{n+1} \C_{1\cdots n+1}] = \\
=[1\wedge 1-\sum_{i=1}^{n}\C_{i}\wedge 1+\sum_{1\leq i_1<i_2\leq n}\C_{i_1i_2}\wedge 1+\cdots+(-1)^{n} \C_{1\cdots n}\wedge 1] + \\ - [1 \wedge \C_{n+1}-\sum_{i=1}^{n}\C_{i}\wedge \C_{n+1}+\sum_{1\leq i_1<i_2\leq n}\C_{i_1i_2}\wedge \C_{n+1}+\cdots+(-1)^{n+1} \C_{1\cdots n}\wedge \C_{n+1}]=\\
=[1-\sum_{i=1}^{n}\C_{i}+\sum_{1\leq i_1<i_2\leq n}\C_{i_1i_2}+\cdots+(-1)^{n} \C_{1\cdots n}]\wedge 1+\\
-[1-\sum_{i=1}^{n}\C_{i}+\sum_{1\leq i_1<i_2\leq n}\C_{i_1i_2}+\cdots+(-1)^{n} \C_{1\cdots n}]\wedge \C_{n+1}.
\end{array} 
\]
\end{proof}
We remark that, by the relation $\D_{1\cdots n}=1-\C_{\no{1} \cdots \no{n}}$, the inclusion-exclusion formula also follows by directly computing $\C_{\no{1} \cdots \no{n}}$  by means of the  distributivity property, as shown below.
\begin{equation}\label{EQ:DISTRIBPROPBIS}
\begin{array}{ll}
\C_{\no{1}\cdots \no{n}}=
\bigwedge_{i=1}^n\C_{\no{i}}=
\bigwedge_{i=1}^n(1-\C_i)=
 (1-\C_1-\C_2+\C_{12})\wedge (1-\C_3)\wedge\cdots \wedge  (1-\C_n)=\\
 =
  (1-\C_1-\C_2-\C_3+\C_{12}+\C_{13}+\C_{23}-\C_{123})\wedge (1-\C_4)\wedge\cdots \wedge  (1-\C_n)=
 \\
= \cdots = 1-\sum_{i=1}^{n}\C_{i}+\sum_{1\leq i_1<i_2\leq n}\C _{i_1i_2}+\cdots+(-1)^{n} \C_{1\cdots n}.
\end{array}
\end{equation}
Then, 
for each  nonempty subset $  \{i_1,\ldots,i_h\}\subseteq \{1,\ldots,n\}$, 
by taking into account (\ref{EQ:DISTRIBGENERAL}) it holds that
\begin{equation}\label{EQ:DISTRIBGEN}
\begin{array}{ll}
\C_{i_1\cdots i_h\no{i_{h+1}}\cdots \no{i_{n}}}=
 \C_{\no{i_{h+1}}\cdots \no{i_{n}}}\wedge\C_{i_1\cdots i_h}=\\
=[1-\sum_{j\in\{i_{h+1},\ldots,i_n\}}\C_{j}+\sum_{\{j_1,j_2\}\subseteq \{i_{h+1},\ldots,i_n\}}\C_{j_1j_2}+\cdots+(-1)^{n-h} \C_{i_{h+1}\cdots i_n}]\wedge \C_{i_1\cdots i_h}=
\\
=\C_{i_1\cdots i_h}-\sum_{j\in\{i_{h+1},\ldots,i_n\}}\C_{i_1\cdots i_hj}+\sum_{\{j_1,j_2\}\subseteq \{i_{h+1},\ldots,i_n\}}\C_{i_1\cdots i_hj_1j_2}+\cdots+(-1)^{n-h} \C_{i_1\cdots i_n}.
\end{array}
\end{equation}
When ${\{i_1,\ldots,i_h\}=\emptyset}$, that is $h=0$, formula 
(\ref{EQ:DISTRIBGEN}) continues to hold because, 
if we set by  convention that  $\C_{\{i_1,\ldots,i_h\}}=\C_{\emptyset}=1$, it reduces to  formula (\ref{EQ:INCLESCLCONG}). Then, in general, it holds that
\begin{equation}\label{EQ:DISTRIBGENBIS}
\C_{i_1\cdots i_h\no{i_{h+1}}\cdots \no{i_{n}}}=\sum_{k=0}^{n-h} (-1)^{k}\sum_{\{j_1,\ldots,j_k\}\subseteq \{i_{h+1},\ldots,i_n\}}\C_{i_1\cdots i_hj_1\cdots j_k},  \;\;  \{i_1,\ldots, i_h\}\subseteq\{1,\ldots,n\}.
\end{equation} 
\begin{remark}
 We observe that, concerning the probabilistic aspects, by recalling (\ref{EQ:LOGPROB}) from  coherence it holds that
\begin{equation}\label{EQ:DISTRIBGENPREV}
 \begin{array}{ll}
 x_{i_1\cdots i_h\no{i_{h+1}}\cdots \no{i_{n}}}=
 x_{i_1\cdots i_h}-\sum_{j\in\{i_{h+1},\ldots,i_n\}}x_{i_1\cdots i_hj}+\sum_{\{j_1,j_2\}\subseteq \{i_{h+1},\ldots,i_n\}}x_{i_1\cdots i_hj_1j_2}+\cdots+(-1)^{n-h} x_{i_1\cdots i_n}=\\
 =\sum_{k=0}^{n-h} (-1)^{k}\sum_{\{j_1,\ldots,j_k\}\subseteq \{i_{h+1},\ldots,i_n\}}x_{i_1\cdots i_hj_1\cdots j_k},
 \end{array}
 \end{equation}
where by convention we set $  x_{i_1\cdots i_h}=x_{\emptyset}=1$ when ${\{i_1,\ldots,i_h\}=\emptyset}$.
Moreover, as each conditional constituent $\C_{i_1\cdots i_h\no{i_{h+1}}\cdots \no{i_{n}}}$ is a nonnegative conditional random quantity, by coherence it must be 
\begin{equation}\label{EQ:CONDNEC}
x_{i_1\cdots i_h\no{i_{h+1}}\cdots \no{i_{n}}}=\sum_{k=0}^{n-h} (-1)^{k}\sum_{\{j_1,\ldots,j_k\}\subseteq \{i_{h+1},\ldots,i_n\}}x_{i_1\cdots i_hj_1\cdots j_k}\geq 0, \;\; \{i_1,\ldots,i_h\}\subseteq\{1,\ldots,n\},
\end{equation}
with  
\begin{equation}\label{EQ:SUM1}
\sum_{\{i_1,\ldots,i_h\}\subseteq\{1,\ldots,n\}} x_{i_1\cdots i_h\no{i_{h+1}}\cdots \no{i_{n}}}=
\sum_{\{i_1,\ldots,i_h\}\subseteq\{1,\ldots,n\}} \sum_{k=0}^{n-h} (-1)^{k}\sum_{\{j_1,\ldots,j_k\}\subseteq \{i_{h+1},\ldots,i_n\}}x_{i_1\cdots i_hj_1\cdots j_k}=1, \;\; 
\end{equation}
as it also follows by observing that 
$\sum_{\{i_1,\ldots,i_h\}\subseteq\{1,\ldots,n\}} \C_{i_1\cdots i_h\no{i_{h+1}}\cdots \no{i_{n}}}=1$. 

Notice that, given a coherent prevision assessment 
$(x_{i_1\cdots i_h}; \emptyset \neq \{i_1,\ldots, i_h\} \subseteq \{1,\ldots, n\})$ on the family $\{\C_{i_1\cdots i_h}; \emptyset \neq \{i_1,\ldots, i_h\} \subseteq \{1,\ldots, n\} \}$, where $\C_{i_1\cdots i_h}=(E_{i_1}|H_{i_1}) \wedge \cdots \wedge (E_{i_h}|H_{i_h})$, as shown by formula (\ref{EQ:CONDNEC}) for every nonempty subset $\{i_1,\ldots, i_h\}$ there exists a unique coherent extension $x_{i_1\cdots i_h\no{i_{h+1}}\cdots \no{i_{n}}}$ for the prevision of the conditional constituent $\C_{i_1\cdots i_h\no{i_{h+1}}\cdots \no{i_{n}}}$.  
\end{remark}
\section{Necessary and sufficient conditions for coherence}
\label{SEC:NECSUFF}
In this section we obtain, under logical independence, two necessary and sufficient coherence conditions. Let  a family of  $n$ conditional events $\mathscr{E}=\{E_1|H_1, \ldots, E_n|H_n\}$ be given, with $E_1,\ldots,E_n, H_1,\ldots,H_n$ logically independent. We denote by  $\M=(x_{ S}:\emptyset \neq S \subseteq\{1,\ldots,n\})$    a prevision assessment on $\mathscr{F}=\{\C_{ S}:\emptyset \neq S \subseteq\{1,\ldots,n\}\}$,
where $\mathscr{C}_{ S}=\bigwedge_{i\in S} (E_i|H_i)
$
and  $x_{ S}=\prev(\C_{ S})$. 
We observe that  $\mathscr{F}$ is the family of all $2^{n}-1$ possible conjunctions among the conditional events in $\mathscr{E}$.
 The first  condition characterizes the coherence of  $\M$ and will be represented in geometrical terms by a suitable convex hull.
The second condition characterizes the coherence of a prevision assessment on  $\mathscr{F}\cup \mathscr{K}$, where $\mathscr{K}$ is the set of conditional constituents associated with   $\mathscr{E}$.

We denote by $C_0,C_1,\ldots C_{3^n-1}$, the  constituents  associated with the family  $\mathscr{E}$,  that is the elements of the partition of $\Omega$ obtained by expanding the expression
\[
\bigwedge_{i=1}^n (E_iH_i \vee \no{E}_iH_i \vee \no{H}_i),
\]
where 
$C_{0}=\no{H}_1\cdots \no{H}_n$.
With each $C_h$ we associate a point 
\begin{equation}\label{EQ:QHGEN}
Q_h=(q_{h S}:\emptyset \neq S \subseteq\{1,\ldots,n\}),
\end{equation}
 where $q_{h S}$ is the value of $\C_{ S}$
when $C_h$ is true. In particular with $C_0$ it is associated  $Q_0=\M$.
We notice that $Q_h$ is the value of the random vector $(\C_{ S}:\emptyset \neq S \subseteq\{1,\ldots,n\})$ when $C_h$ is true.
By discarding $Q_0$, we denote by $\Q$ the set of remaining points $Q_{h}$'s associated with the pair $(\mathscr{F},\M)$ and by $\I_{\Q}$ the convex hull of the set $\Q$. We denote by $\B$ the subset of  $\Q$, constituted by $2^n$ binary points  $Q_1,\ldots, Q_{2^n}$,
defined as
\begin{equation}\label{EQ:SETB}
\B=\{Q_1,\ldots, Q_{2^n}\}=\{Q_h\in \Q: q_{h S}\in\{0,1\},\; S=\{i\},\; i=1,\ldots,n\}.
\end{equation}
We observe that the points $Q_1,\ldots, Q_{2^n}$ are associated with the $2^n$ constituents $C_h$'s obtained by  expanding the expression
\[
\bigwedge_{i=1}^n (E_iH_i \vee \no{E}_iH_i),
\]
which coincides with $\bigwedge_{i=1}^n H_i$.
Notice that, 
given any $C_h$ such that  $Q_h\in \B$, the sub-vector  $(q_{hS},S=\{i\},\; i=1,\ldots,n)$ is  a vertex of the unit hypercube $[0,1]^n$ and it is the value assumed by the random vector
 $(E_1|H_1,\cdots , E_n|H_n)$ when  $C_h$ is true. We also remark that, from the definition of conjunction it follows that 
\begin{equation}\label{EQ:PROPB}
Q_h\in \B \Longrightarrow\;\; q_{h S}\in\{0,1\},\;\;\forall\,\,\emptyset \neq S     \subseteq\{1,\ldots,n\}.
\end{equation}
Then, the set $\B$ can be equivalently defined as
\[
\B=\{Q_h\in \Q: q_{h S}\in\{0,1\},\;\emptyset \neq S     \subseteq\{1,\ldots,n\}\}.
\]
We denote by $\I_{\B}$ the convex hull of the set $\B$; of course $\I_{\B} \subseteq \I_{\Q}$. Then we have
\begin{theorem}\label{THM:CNESIB}
Given a family of $n$ conditional events $\mathscr{E}=\{E_1|H_1,\ldots,E_n|H_n\}$, 
let $\M=(x_{ S}:\emptyset \neq S \subseteq\{1,\ldots,n\})$ be  a prevision assessment on the family
	 $\mathscr{F}=\{\C_{ S}:\emptyset \neq S \subseteq\{1,\ldots,n\}\}$, where $\mathscr{C}_{ S}=\bigwedge_{i\in S} (E_i|H_i)$.
	Under the assumption of logical independence of $E_1,\ldots,E_n, H_1,\ldots,H_n$, the prevision assessment $\M$ on $\mathscr{F}$ is coherent if and only if $\mathcal{M}$ belongs to the convex hull $\mathcal{I}_{\mathcal{B}}$ of the $2^n$ binary points  $Q_1,\ldots, Q_{2^n}$.
\end{theorem}
\begin{proof}
$(\Rightarrow)$ Assume that $\M$ is coherent. 
Then, all the inequalities in (\ref{EQ:CONDNEC}) are satisfied. 
We observe that the condition  $\M\in\I_{\B}$ is satisfied if there exist suitable nonnegative coefficients $\lambda_h$'s, with $\sum_{h=1}^{2^n}\lambda_h=1$, such that 
$\M=\sum_{h=1}^{2^n}\lambda_hQ_{h}$. This means that for each component $x_{ S}$ of $\M$ it must be
$x_{ S}=\sum_{h=1}^{2^n}\lambda_hq_{h S}=\sum_{h:q_{h S}=1}\lambda_h$. We observe that with each $Q_h\in \B$ it is associated a unique subset $\{i_1,\ldots,i_k\} \subseteq \{1,\ldots,n\}$ such that, when $S=\{i\}$, $i=1,\ldots,n$, it holds that $q_{hS}=q_{h\{i\}}=1$ if $i \in \{i_1,\ldots,i_k\}$ and $q_{hS}=q_{h\{i\}}=0$ if $i \in \{i_{k+1},\cdots, i_{n}\}=\{1,\ldots,n\} \setminus \{i_1,\ldots, i_k\}$. Then, by changing notations,  the point $Q_h$ associated with $\{i_1,\ldots,i_k\}$  will be denoted by the symbol $Q_{i_1\cdots i_k\no{i_{k+1}}\cdots \no{i_{n}}}$ and the coefficient $\lambda_h$ will be denoted by $\lambda_{i_1\cdots i_k\no{i_{k+1}}\cdots \no{i_{n}}}$.
By this change of notations, the binary quantity $q_{hS}$ becomes $q_{i_1\cdots i_k \no{i_{k+1}} \cdots \no{i_n}S}$, with 
\[
q_{i_1\cdots i_k \no{i_{k+1}} \cdots \no{i_n}S}=
\left\{
\begin{array}{ll}
1, & \mbox{ if }  S\subseteq \{i_1,\ldots,i_k\},\\
0, & \mbox{ if }  S\nsubseteq \{i_1,\ldots,i_k\}.\\
\end{array}
\right.
\]
Then the equality $x_{ S}=\sum_{h=1}^{2^n}\lambda_hq_{h S}$ becomes $x_{ S}
=\sum_{\{i_1,\ldots, i_k\}\supseteq  S}
\lambda_{i_1\cdots i_k \no{i_{k+1}} \cdots \no{i_n}}$.
Then, more  explicitly, the condition $\M\in\I_{\B}$ is satisfied if there exists a vector, with  components,  $\Lambda=(\lambda_{i_1\cdots i_k \no{i_{k+1}} \cdots \no{i_n}}; \{i_1,\ldots,i_k\} \subseteq \{1,\ldots,n\})$ which is a solution of the system  below.
\begin{equation}{\label{EQ:SYSTEM}}
(\Sigma_{\B})\left\{
\begin{array}{ll}
\begin{array}{ll}
x_{ S}
=\sum_{\{i_1,\ldots, i_k\}\supseteq  S}
\lambda_{i_1\cdots i_k \no{i_{k+1}} \cdots \no{i_n}},\;\; \;  \emptyset \neq S\subseteq \{1,2\ldots,n\},
\end{array}\\
\sum_{\{i_1,\ldots, i_k\} \subseteq \{1,2\ldots,n\}}\lambda_{i_1\cdots i_k \no{i_{k+1}} \cdots \no{i_n}}=1,\\
\lambda_{i_1\cdots i_k \no{i_{k+1}} \cdots \no{i_n}}\geq 0, \; \; \; \{i_1,\ldots, i_k\} \subseteq \{1,2\ldots,n\}.
\end{array}
\right.
\end{equation}
We observe that $\Lambda$ has $2^n$ (nonnegative) components and  $(\Sigma_{\B})$ has $2^n$ equations. 
By coherence of $\M$, from (\ref{EQ:CONDNEC}) and (\ref{EQ:SUM1}) we can compute the quantities 
$x_{i_1\cdots i_k \no{i_{k+1}} \cdots \no{i_n}}$, for all $\{i_1,\ldots,i_k\} \subseteq \{1,\ldots,n\}$, which are nonnegative and with their sum equal to 1. 
Moreover, by  (\ref{EQ:XS}), for each subset $S$ it holds that  $x_{ S}
=\sum_{\{i_1,\ldots, i_h\}\supseteq  S}
x_{i_1\cdots i_h \no{i_{h+1}} \cdots \no{i_n}}$, which has the same structure of 
 the equation  $x_{ S}=\sum_{\{i_1,\ldots, i_h\}\supseteq  S}
\lambda_{i_1\cdots i_h \no{i_{h+1}} \cdots \no{i_n}}$ in $(\Sigma_{\B})$. Then, 
$(\Sigma_{\B})$ is solvable  and the  (unique) solution is 
 the  vector $\Lambda$ with components 
\[
\lambda_{i_1\cdots i_k \no{i_{k+1}} \cdots \no{i_n}}=x_{i_1\cdots i_k \no{i_{k+1}} \cdots \no{i_n}}=\prev(\C_{i_1\cdots i_k \no{i_{k+1}} \cdots \no{i_n}}), \;\, \forall \{i_1,\ldots,i_k\} \subseteq \{1,\ldots,n\}.
\]
Thus, the condition $\M \in \I_{\B}$ is satisfied.  \\
$(\Leftarrow)$ Assume that  $\M \in \I_{\B}$. Then, $\M \in \I_{\Q}$ because $\B \subset \Q$ and hence  the  system $(\Sigma)$  is solvable. Moreover, $\I_0=\emptyset$ because all the coefficients $\lambda_{i_1\cdots i_k\no{i_{k+1}} \cdots \no{i_n}}$'s are associated with the constituents, which we denote by $C_{i_1\cdots i_k\no{i_{k+1}} \cdots \no{i_n}}$'s, such that for every subset $\{i_1,\ldots,i_k\}$ it holds that $C_{i_1\cdots i_k\no{i_{k+1}} \cdots \no{i_n}} \subseteq H_i$, for every $i=1,\ldots,n$. Thus, by Theorem   \ref{CNES-PREV-I_0-INT} the prevision assessment $\M$ is coherent. 
\end{proof}
\begin{remark}\label{REM:CNESIB}
As shown by Theorem \ref{THM:CNESIB}, under logical independence of the basic events $E_1,\ldots E_n, H_1,\ldots, H_n$, the coherence of $\M$ amounts to the solvability of  system $(\Sigma_{\B})$.  Moreover, 
 $(\Sigma_{\B})$ is solvable if and only if the following  inequalities  are satisfied
\begin{equation}\label{EQ:LISTINEQ}
\sum_{k=0}^{n-h} (-1)^{k}\sum_{\{j_1,\ldots,j_k\}\subseteq \{i_{h+1},\ldots,i_n\}}x_{i_1\cdots i_hj_1\cdots j_k}\geq 0, \;\;\forall\;\; \{i_1,\ldots,i_h\}\subseteq\{1,\ldots,n\},
\end{equation}
where we recall that,  by  (\ref{EQ:CONDNEC}), the first member of (\ref{EQ:LISTINEQ}) is the prevision $x_{i_1\cdots i_h\no{i_{h+1}}\cdots \no{i_{n}}}$ of  $\C_{i_1\cdots i_h\no{i_{h+1}}\cdots \no{i_{n}}}$.
Therefore,  under logical independence, 
the set of all coherent assessments on the  family $\mathscr{F}$ is  the set of
assessments $\M$, with  components $x_S$ which satisfy the  list of linear inequalities (\ref{EQ:LISTINEQ}). Indeed, $x_{i_1\cdots i_hj_1\cdots j_k}$ coincides with the component $x_S$, where $S=\{i_1,\ldots,i_h,j_1,\ldots,  j_k\}$.  
\end{remark}
We will now give another result on coherence under logical independence.
We denote by $\Delta$   the $(2^n-1)$-dimensional simplex of $\mathbb{R}^{2^n}$, that is the set  of vectors $\mathscr{V}=(v_{\{i_1,\ldots,i_h\}}, \{i_1,\ldots,i_h\}\subseteq\{1,\ldots,n\})$ such that 
\begin{equation*}
\sum_{\{i_1,\ldots,i_h\}\subseteq\{1,2,\ldots,n\}}v_{\{i_1,\ldots,i_h\}}=1, \,\,
v_{\{i_1,\ldots,i_h\}}\geq 0, \,\,\,\, \forall\,  \{i_1,\ldots, i_h\}\subseteq  \{1,\dots,n\}.
\end{equation*}
We observe that, given any 
 $\mathscr{V}\in \Delta$, we can construct a prevision assessment 
 $\M=(x_{ S}:\emptyset \neq S \subseteq\{1,\ldots,n\})$ on
$\mathscr{F}=\{\C_{ S}:\emptyset \neq S \subseteq\{1,\ldots,n\}\}$, where each $x_{S}$ is obtained by applying (\ref{EQ:XS})  with $x_{i_1i_2\cdots i_h \no{i_{h+1}} \cdots \no{i_k}}$ replaced by
$v_{\{i_1,\ldots, i_h\}}$.  
Moreover, concerning the set 
 $\mathscr{K}$ of the  conditional constituents  $\{\C_{i_1\cdots i_h \no{i_{h+1}} \cdots \no{i_n}}:\;\; \{i_1,\cdots, i_h\}\subseteq \{1,\ldots,n\}\}$,  each $\C_{i_1\cdots i_h \no{i_{h+1}} \cdots \no{i_n}}$ is obtained from suitable elements of $\mathscr{F}$ by applying 
(\ref{EQ:DISTRIBGENBIS}).  In this way,  we also obtain a prevision assessment $\mathscr{P}=(x_{i_1\cdots i_h \no{i_{h+1}} \cdots \no{i_n}}, \{i_1,\ldots,i_h\}\subseteq\{1,\ldots,n\})$  on $\mathscr{K}$, with $x_{i_1\cdots i_h \no{i_{h+1}} \cdots \no{i_n}}=v_{\{i_1,\ldots,i_h\}}$; thus $\mathscr{P}=\mathscr{V}$.
The next result shows that, under logical independence, by the simplex $\Delta$  we obtain all the coherent prevision assessments on $\mathscr{F}\cup \K$. For the sake of simplicity, even if both $\M$ and $\mathscr{P}$ contain the element $x_{1\cdots n}=\prev(\C_{1\cdots n})$, we denote by $(\M, \mathscr{P})=(\M, \mathscr{V})$ the prevision assessment on $\mathscr{F}\cup \mathscr{K}$ associated with $\mathscr{V}$.
\begin{theorem}\label{THM:SIMPLESSO} Let a 
 family of $n$ conditional events $\mathscr{E}=\{E_1|H_1,\ldots,E_n|H_n\}$ be given, with $E_1,\ldots,E_n, H_1,\ldots,H_n$ logically independent. A prevision assessment  $(\M,\mathscr{P})$ on $\mathscr{F}\cup \mathscr{K}$ is coherent if and only if it is associated to a vector $\mathscr{V}\in \Delta$.
\end{theorem}
\begin{proof}
	If $(\M,\mathscr{P})$ is a coherent prevision assessment on $\mathscr{F}\cup \mathscr{K}$, then $\M$ is obtained from $\mathscr{P}$ by means of (\ref{EQ:XS}) and from (\ref{EQ:CONBAS}) it holds that $\mathscr{P}\in\Delta$. Then
		 $(\M,\mathscr{P})$ is associated with the vector 
	$\mathscr{V}=\mathscr{P}\in\Delta$. 
	
	Conversely, if $(\M,\mathscr{P})$ is associated to some $\mathscr{V}\in \Delta$, then $\mathscr{P}=\mathscr{V}$. Moreover, as shown in the proof of Theorem \ref{THM:CNESIB}, under logical independence, by setting  	
	\[
	\lambda_{i_1\cdots i_h \no{i_{h+1}} \cdots \no{i_n}}=x_{i_1\cdots i_h \no{i_{h+1}} \cdots \no{i_n}}=v_{\{i_1,\ldots, i_h\}},\;\; \forall \{i_1,\ldots,i_h\} \subseteq \{1,\ldots,n\},
	\]
	the system $(\Sigma_{\B})$ is solvable, that is $\M\in \I_{\B}$; thus $\M$ is coherent. Finally, if  we  extend the assessment $\M$, defined on $\mathscr{F}$, to the family $\K$, by recalling (\ref{EQ:XS}) and  (\ref{EQ:CONDNEC}) the extension coincides with  $\mathscr{P}$. Hence, $(\M,{\mathscr{P}})$ is coherent.
\end{proof}
\section{Some further  aspects}
\label{EQ:FURTHASP}
In this section we examine some further aspects which are related with Theorem \ref{THM:CNESIB}. 
We observe that  each  $Q_h$ defined as in (\ref{EQ:QHGEN}) is itself a prevision assessment  on $\mathscr{F}$ and hence, by Theorem \ref{THM:CNESIB}, $Q_h$ is coherent if and only if $Q_h\in \I_{\B}$. 
In the next result we prove that, under logical independence, 
coherence of $\M$ requires coherence of 
all the points $Q_h$'s. In other words, if  $\M$ is coherent, then  for each $h$ it holds that  $Q_h\in \I_{\B}$, even if $Q_h\notin \B$. 
\begin{theorem}
	Let  $n$ conditional events $E_1|H_1,\ldots, E_{n}|H_{n}$ be given, with $E_1,H_1,\ldots,E_{n},H_{n}$ logically independent. 
	Given a coherent prevision assessment $\M=(x_{ S}:\emptyset \neq S \subseteq\{1,\ldots,n\})$ on the family
	 $\mathscr{F}=\{\C_{ S}:\emptyset \neq S \subseteq\{1,\ldots,n\}\}$, for every point $Q_h$ it holds that  $Q_h$ is a coherent assessment on $\mathscr{F}$, or equivalently $Q_h\in \I_{\B}$.
\end{theorem}	
\begin{proof}
	Of course, for each $Q_h\in \B$  it holds that $Q_h\in \I_{\B}$, that is $Q_h$ is coherent. 
	Let us consider any point $Q_h$, $h\neq 0$, associated with a constituent $C_h$, such that $Q_h\notin \B$.
	Without loss of generality we assume that
	\[
	C_h\subset \no{H}_1\cdots \no{H}_k {H}_{k+1}\cdots {H}_{n},\;\; 1\leq k<n.
	\]	
	Then,  for the components $q_{h S}$ of $Q_h$, with $ S=\{i\}$, $i=1,\ldots,n$, it holds that 
	\[
	q_{h S}=
	\left\{
	\begin{array}{ll}
	x_{i}, \mbox{ if }  S=\{i\}, i=1,\ldots,k;\\
	b_{i}\in\{0,1\}, \mbox{ if }  S=\{i\}, i=k+1,\ldots,n.
	\end{array}
	\right.
	\]	
	More in general  we have
	\begin{equation}\label{EQ:q_h-A}
	q_{h S}=\left\{\begin{array}{ll}
	x_{ S}, &\mbox{ if }  S\subseteq \{1,\ldots,k\},\\
	1, &\mbox{ if }  S\subseteq \{k+1,\ldots,n\} \mbox{ and } b_i=1, \forall \, i\in  S ,\\
	0, &\mbox{ if }  S\subseteq \{k+1,\ldots,n\} \mbox{ and } b_i=0, \mbox{ for some } \,i\in  S ,
	\\
	x_{S'} &\mbox{ if }  S\cap \{1,\ldots,k\}= S'\neq \emptyset \mbox{ and } b_i=1, \forall \, i\in 
	 S\setminus  S'\neq \emptyset ,\\
	0 &\mbox{ if }  S\cap \{1,\ldots,k\}= S'\neq \emptyset \mbox{ and } b_i=0,  \mbox{for some } \, i\in 
	 S\setminus  S'\neq \emptyset .\\
	\end{array}
	\right.
	\end{equation}
		We denote by $\M_k$ the sub-assessment of $\M$ defined as
	\[
	\M_k=(x_{ S}:\emptyset \neq S \subseteq\{1,\ldots,k\})
	\]
	on the sub-family $\mathscr{F}_k$ of $\mathscr{F}$ defined as
	\[
	\mathscr{F}_k=(\C_{ S}:\emptyset \neq S \subseteq\{1,\ldots,k\}).
	\]
	The coherence of $\M$ implies the coherence of the sub-assessment  $\M_k$. 
	We observe that, as the events $E_i,H_i$, $i=1,\ldots,n$ are logically independent,  the extension $\M^*_k$ of $\M_k$ on $\mathscr{F}_{k}\cup\{E_i|H_i, i=k+1,\ldots,n\}$, such that $P(E_i|H_i)=b_i\in\{0,1\}$ for $i=k+1,\ldots,n$, is coherent.
	Moreover, there exists
	a unique  extension  $\M^*$ of $\M^*_k$ on the family $\mathscr{F}$ because, for the assessment 
	\[
	\M^*=(x^*_{ S}:\emptyset \neq S \subseteq\{1,\ldots,n\}),
	\]
	each component $x^*_{ S}$ is uniquely determined by $\M_k^*$. Indeed, it holds that
	\begin{equation}\label{EQ:x*S-A}
	x^*_{ S}=\left\{\begin{array}{ll}
	x_{ S}, &\mbox{ if }  S\subseteq \{1,\ldots,k\},\\
	1, &\mbox{ if }  S\subseteq \{k+1,\ldots,n\} \mbox{ and } b_i=1, \forall \, i\in  S ,\\
	0, &\mbox{ if }  S\subseteq \{k+1,\ldots,n\} \mbox{ and } b_i=0, \mbox{ for some } \,i\in  S ,
	\\
	x_{S'} &\mbox{ if }  S\cap \{1,\ldots,k\}= S'\neq \emptyset \mbox{ and } b_i=1, \forall \, i\in 
	S''= S\setminus  S'\neq \emptyset ,\\
	0 &\mbox{ if }  S\cap \{1,\ldots,k\}= S'\neq \emptyset \mbox{ and } b_i=0,  \mbox{for some } \, i\in 
	S''= S\setminus  S'\neq \emptyset .\\
	\end{array}
	\right.
	\end{equation}
	The uniqueness of the extension $x^*_{ S}=0$, or
	$x^*_{ S}=1$ shown in the second and third lines of (\ref{EQ:x*S-A}),  follows from Theorem \ref{THM:TEOREMAAI13}. Moreover, the uniqueness of the extension $x^*_{ S}=x_{ S'}$ follows because, by Theorem \ref{THM:RISPREL}, it holds that 
	\begin{equation}\label{EQ:ANDCONG}
	\max \{x_{ S'}+x_{ S''}-1,0\} \leq x^*_{ S} \leq \min\{x_{ S'},x_{ S''}\},
	\end{equation}
and, from Theorem \ref{THM:TEOREMAAI13}, it holds that $x_{ S''}=\prev(\C_{ S\setminus  S'})=1$; thus (\ref{EQ:ANDCONG}) becomes 
\[	
\max \{x_{ S'}+x_{ S''}-1,0\}=x_{S'} \leq x^*_{ S} \leq x_{S'}= \min\{x_{ S'},x_{ S''}\}.
\]
Finally,  the uniqueness of the extension $x^*_{ S}=0$ in the last line of (\ref{EQ:x*S-A}) follows because,  from Theorem \ref{THM:TEOREMAAI13}, it holds that $x_{ S''}=\prev(\C_{ S\setminus  S'})=0$; thus (\ref{EQ:ANDCONG}) becomes 
\[	
\max \{x_{ S'}+x_{ S''}-1,0\}=0 \leq x^*_{ S} \leq 0= \min\{x_{ S'},x_{ S''}\}.
\]	Of course, as the extension  $\M^*$ of $\M^*_k$ is unique, coherence of $\M^*_k$ implies coherence of $\M^*$. Then by Theorem \ref{THM:CNESIB}, it holds that $\M^*\in \I_{\B}$.
	Finally, from (\ref{EQ:q_h-A}) and (\ref{EQ:x*S-A})  it follows that  $x^*_{ S}=q_{h S}$ $\forall S \neq \emptyset$. Therefore $\M^*=Q_h$, so that $Q_h$ is coherent, or equivalently $Q_h \in \I_{\B}$.
\end{proof}	
\begin{remark}
	We recall that each $Q_h$ associated with the pair $(\mathscr{F},\M)$ represents the value of the random vector 
	$(C_{ S}:\emptyset \neq S \subseteq\{1,\ldots,n\})$ when $C_h$ is true. Then coherence of $\M$ implies that, as for the case of unconditional events, each possible value $Q_h$ of the random vector is itself a particular coherent assessment on $\mathscr{F}$.
\end{remark}
\section{Some examples and counterexamples}\label{SEC:EX}
As shown by Theorem \ref{THM:CNESIB}, under logical independence of $E_1,\ldots,E_n, H_1,\ldots,H_n$, coherence of $\M$ amounts to condition $\M \in \I_{\B}$, that is to validity of all inequalities in formula (\ref{EQ:CONDNEC}). We examine this aspect for $n=2$ and $n=3$ in the examples below.
\begin{example}\label{EX:CNESIBn=2}
	In this example we obtain the lower and upper bounds given in  Theorem \ref{THM:FRECHET} by using the conditional constituents.
We consider $\mathscr{E}=\{E_1|H_1,E_2|H_2\}$ and $\mathscr{F}=\{E_1|H_1, E_2|H_2, (E_1|H_1) \wedge (E_2|H_2)\}$, with $E_1,E_2,H_1,H_2$ logically independent. Then, 	
let $\M=(x_1,x_2,x_{12})$ be a prevision assessment on $\mathscr{F}$. The set of conditional constituents is $\mathscr{K}=\{
\C_{12}, \C_{1\no{2}}, \C_{\no{1}2}, \C_{\no{1}\,\no{2}}\}$, where
$\C_{12}=(E_1|H_1) \wedge (E_2|H_2),\;\;
\C_{1\no{2}}=(E_1|H_1) \wedge (\no{E}_2|H_2),\;\;
\C_{\no{1}2}=(\no{E}_1|H_1) \wedge (E_2|H_2),\;\;
\C_{\no{1}\,\no{2}}=(\no{E}_1|H_1) \wedge (\no{E}_2|H_2).
$
As made in the proof of Theorem \ref{THM:CNESIB}, we change notations for the points  
 $Q_h$'s of the set 
 $\B$. In this example $n=2$, then
 $\B=\{
 Q_{12}, Q_{1\no{2}}, Q_{\no{1}2}, Q_{\no{1}\,\no{2}}\}$, where
\[
Q_{12}=(1,1,1),\;  Q_{1\no{2}}=(1,0,0),\; Q_{\no{1}2}=(0,1,0),\; Q_{\no{1}\,\no{2}}=(0,0,0).
\]
The previsions of the conditional constituents $\C_{12}, \C_{1\no{2}}, \C_{\no{1}2}, \C_{\no{1}\,\no{2}}$ are, respectively,
\[
x_{12},\; x_{1\no{2}}=x_1-x_{12},\; x_{\no{1}2}=x_2-x_{12},\; x_{\no{1}\,\no{2}}=1-x_1-x_2+x_{12}.
\]
These previsions are the coefficients which allow to represent $\M$ as a linear convex combinations of the points of the set $\B$. By Remark \ref{REM:CNESIB}, coherence of $\M$ amounts to the inequalities in (\ref{EQ:LISTINEQ}), that is  
\[
x_{12} \geq 0,\; x_1-x_{12} \geq 0,\; x_2-x_{12} \geq 0,\; 1-x_1-x_2+x_{12} \geq 0,
\]
which are equivalent to the following conditions
\[
(x_1,x_2) \in [0,1]^2,\;\;\; \max\{0,x_1+x_2-1\}\; \leq \; x_{12} \; \leq \; \min \{x_1,x_2\}.
\]
Notice that, by recalling Theorem \ref{THM:SIMPLESSO}, each vector $\mathscr{V}$ of the 3-dimensional simplex $\Delta$ determines a coherent prevision assessments on $\mathscr{F}\cup\mathscr{K}=\{E_1|H_1,E_2|H_2,\mathscr{C}_{12},\mathscr{C}_{1\no{2}},\mathscr{C}_{\no{1}2},\mathscr{C}_{\no{1}\,\no{2}} \}$. For instance, with the vector $\mathscr{V}=(v_{\{1,2\}},v_{\{1\}},v_{\{2\}},v_{\emptyset})=(\frac{1}{6},\frac{1}{6},\frac{1}{3},\frac{1}{3})$ it is associated the assessment $(\M,\mathscr{P})=
(x_1,x_2,x_{12},x_{1\no{2}},x_{\no{1}2},x_{\no{1}\,\no{2}})$, where 
\[
\begin{array}{ll}
x_{1}=v_{\{1,2\}}+v_{\{1\}}=\frac13,\;\;
x_{2}=v_{\{1,2\}}+v_{\{2\}}=\frac12,\;\;
x_{12}=v_{\{1,2\}}=\frac16,\\
x_{1\no{2}}=v_{\{1\}}=\frac16,\;\;
x_{\no{1}2}=v_{\{2\}}=\frac13,\;\;
x_{\no{1}\,\no{2}}=v_{\{\emptyset\}}=\frac13.
\end{array}
\]
\end{example}
\begin{example}\label{EX:CNESIBn=3}
In this example, by using    the set of conditional constituents,	we obtain the same  result   given in \cite[Corollary 1]{GiSa19}.	We start by a family	
$\mathscr{E}=\{E_1|H_1,E_2|H_2,E_3|H_3\}$, where  $E_1,E_2,E_3,H_1,H_2,H_3$ are logically independent.
The  conditional constituents are $\C_{123}=(E_1|H_1) \wedge (E_2|H_2) \wedge (E_3|H_3), \ldots, \C_{\no{1}\,\no{2}\,\no{3}}=(\no{E}_1|H_1) \wedge (\no{E}_2|H_2) \wedge (\no{E}_3|H_3)
$.
Let $\M=(x_1,x_2,x_3,x_{12},x_{13},x_{23},x_{123})$ be a prevision assessment on the family  $\mathscr{F}=\{E_1|H_1, E_2|H_2, E_3|H_3, \C_{12}, \C_{13}, \C_{23}, \C_{123}\}$,
where $\C_{ij}=E_i|H_i \wedge E_j|H_j$. By logical independence and by  Theorem \ref{THM:CNESIB},  coherence of $\M$ amounts to the condition $\M \in \I_{\B}$, where  $\B$ is the set of points
\[
\begin{array}{ll}
Q_{123}=(1,1,1,1,1,1),\; Q_{12\no{3}}=(1,1,0,1,0,0,0),\;
 Q_{1\no{2}3}=(1,0,1,0,1,0,0),\;
 \\
 Q_{1\no{2}\,\no{3}}=(1,0,0,0,0,0,0),\;
Q_{\no{1}23}=(0,1,1,0,0,1),\; Q_{\no{1}2\no{3}}=(0,1,0,0,0,0,0),\;
\\
 Q_{\no{1}\,\no{2}3}=(0,0,1,0,0,0,0),\; Q_{\no{1}\,\no{2}\,\no{3}}=(0,0,0,0,0,0,0).
 \end{array}
\]
By recalling  (\ref{EQ:CONDNEC}), the previsions of the conditional constituents,  which are the coefficients in the representation of   $\M$ as a linear convex combinations of the points of the set $\B$,  are 
\[
x_{123},\; x_{12\no{3}}=x_{12}-x_{123},\; x_{1\no{2}3}=x_{13}-x_{123},\; x_{1\no{2}\,\no{3}}=x_{1\no{2}}-x_{1\no{2}3}=x_1-x_{12}-x_{13}+x_{123},
\]
\[
x_{\no{1}23}=x_{23}-x_{123},\; x_{\no{1}2\no{3}}=x_{\no{1}2}-x_{\no{1}23} =x_2-x_{12}-x_{23}+x_{123},\; x_{\no{1}\,\no{2}3}=x_{\no{1}3}-x_{\no{1}23}=x_3-x_{13}-x_{23}+x_{123},\; 
\]
\[
x_{\no{1}\,\no{2}\,\no{3}}=x_{\no{1}\,\no{2}}-x_{\no{1}\,\no{2}3}=(1-x_1-x_2+x_{12})-(x_3-x_{13}-x_{23}+x_{123})=1-x_1-x_2-x_3+x_{12}+x_{13}+x_{23}-x_{123}.
\]
 By Remark \ref{REM:CNESIB}, coherence of $\M$ amounts to the inequalities in (\ref{EQ:LISTINEQ}), that is 
\[
x_{123} \geq 0,\; x_{12}-x_{123} \geq 0,\; x_{13}-x_{123} \geq 0,\; x_1-x_{12}-x_{13}+x_{123} \geq 0,\; x_{23}-x_{123} \geq 0,
\]
\[
x_2-x_{12}-x_{23}+x_{123} \geq 0,\;  x_3-x_{13}-x_{23}+x_{123} \geq 0,\; 1-x_1-x_2-x_3+x_{12}+x_{13}+x_{23}-x_{123} \geq 0,
\]
which can be written as  $(x_1,x_2,x_3) \in [0,1]^3,\; \; x_{123}' \leq \; x_{123} \; \leq \; x_{123}''$, where
\[
x_{123}'= \max\{0,x_{12}+x_{13}-x_1,x_{12}+x_{23}-x_2,x_{13}+x_{23}-x_3\},
\]
\[
x_{123}''=\min \{x_{12},x_{13},x_{23},1-x_1-x_2-x_3+x_{12}+x_{13}+x_{23}\}.
\]
Notice that there are inequalities which hold, even if they are not evident. Indeed, as $x_{123}''\geq x_{123}'$, it holds for instance that $1-x_1-x_2-x_3+x_{12}+x_{13}+x_{23} \geq x_{12}+x_{13}-1$, that is $x_{23} \geq x_2+x_3-1$, and so on. 
Of course, if  $x_{123}''< x_{123}'$ (because for instance $x_{12}<x_{12}+x_{13}-x_1$, that is $x_{13}>x_1$), then  the assessment is not coherent.
Moreover, by recalling Theorem \ref{THM:SIMPLESSO}, each vector $\mathscr{V}$ of the 7-dimensional simplex $\Delta$ determines a coherent prevision assessments on 
$\mathscr{F}\cup\mathscr{K}$.
\end{example}
We also remark that, in case of some logical dependencies, Theorem \ref{THM:CNESIB} is no more valid; that is, coherence is not equivalent to the condition $\M \in \I_{\B}$. We give below two examples; in the first one the set $\B$ is empty. 
\begin{example}\label{EX:AgHAgK}
Let  three events $A, H, K$ be given, with 
$HK=\emptyset$ and $A$ logically independent of $H$ and $K$. Moreover, let $\M=(x,y,z)$ be a prevision assessment on the family $\mathscr{F}=\{A|H,A|K,(A|H)\wedge(A|K)\}$. The constituents generated by $\{A|H,A|K\}$ are 
\[
C_1=AH\no{K},\; C_2=A\no{H}K,\; C_3=\no{A}H\no{K},\; C_4=\no{A}\no{H}K,\; C_0=\no{H}\no{K}.
\]
The associated points $Q_h$'s for the pair $(\mathscr{F},\M)$ are 
\[
Q_1=(1,y,y),\; Q_2=(x,1,x),\; Q_3=(0,y,0),\; Q_4=(x,0,0),\; Q_0=\M=(x,y,z).
\]
As we can see, it holds that $\B=\emptyset$ and hence $\I_{\B}=\emptyset$; then, to check coherence of $\M$ we cannot use the condition $\M \in \I_{\B}$, which is meaningless. Instead, in order to check coherence we need to start by checking the condition $\M \in \I_{\Q}$, which amounts to solvability of the system below.
\begin{equation*}
\left\{
\begin{array}{l}
x=\lambda_1+\lambda_2 x+\lambda_4 x, \\
y=\lambda_1 y+\lambda_2+\lambda_3 y, \\
z=\lambda_1 y+\lambda_2 x, \\
\lambda_1+\cdots+\lambda_4=1,\;\; \lambda_h \geq 0,\; h=1,2,3,4,
\end{array}
\right.
\end{equation*}
which can be written as 
\begin{equation*}
\left\{
\begin{array}{l}
xy=\lambda_1 y+\lambda_2 xy+\lambda_4 xy, \\
xy=\lambda_1 xy+\lambda_2 x+\lambda_3 xy, \\
z=\lambda_1 y+\lambda_2 x, \\
\lambda_1+\cdots+\lambda_4=1,\;\; \lambda_h \geq 0,\; h=1,2,3,4.
\end{array}
\right.
\end{equation*}
By summing the first two equations, we obtain: $z=xy$; then, the unique coherent extension of $(x,y)$ to the conditional constituent $(A|H)\wedge(A|K)$ is $z=xy$ (see also \cite{GiSa19b}).
 We observe that the assessment $(x,y,z)$ uniquely determines the extensions to the other conditional constituents, $(A|H)\wedge(\no{A}|K)$, $(\no{A}|H)\wedge(A|K)$, and  $(\no{A}|H)\wedge(\no{A}|K)$, given by $x(1-y), (1-x)y$, and $(1-x)(1-y)$, respectively. 
\end{example}
\begin{remark}
Example \ref{EX:AgHAgK} shows that in general, given two conditional events $E_1|H_1, E_2|H_2$,
 in order a prevision assessment $(x_{12},x_{1\no{2}}, x_{\no{1}2}, x_{\no{1}\,\no{2}})$ on the family of conditional constituents $\{\C_{12},\C_{1\no{2}}, \C_{\no{1}2}, \C_{\no{1}\,\no{2}}\}$ be coherent, it is not sufficient that the conditions given in (\ref{EQ:CONBAS}), that is \[
x_{12}+x_{1\no{2}}+x_{\no{1}2}+ x_{\no{1}\,\no{2}}=1,\; x_{12}\geq 0, x_{1\no{2}}\geq 0, x_{\no{1}2}\geq 0, x_{\no{1}\,\no{2}}\geq 0,
\]
be satisfied. Indeed,  even if the previous conditions imply that 
$x_{12}+x_{1\no{2}}=x_1$, and $x_{12}+x_{\no{1}2}=x_2$, in Example \ref{EX:AgHAgK} coherence  also requires that the conditions  $x_{12}= x_1x_2$, $x_{1\no{2}}= x_1(1-x_2)$, $x_{\no{1}2}=(1-x_1)x_2$,  and $x_{\no{1}\,\no{2}}=(1-x_1)(1-x_2)$  be satisfied, and this is not guaranteed.  Then, the assessment $(x_{12}, x_{1\no{2}},x_{\no{1}2},x_{\no{1}\,\no{2}})$ could be incoherent. The same remark holds more in general when we consider the c-constituents associated with $n$ conditional events. In other words, the conditions in (\ref{EQ:CONBAS}) are necessary but not sufficient for coherence.
\end{remark}
\begin{example}
We examine the previous example, by assuming $HK \neq \emptyset$. In this case the constituents generated by $\{A|H,A|K\}$ are 
\[
C_1=AH\no{K},\; C_2=A\no{H}K,\; C_3=\no{A}H\no{K},\; C_4=\no{A}\no{H}K,\; C_5=AHK,\;C_6=\no{A}HK,\; C_0=\no{H}\no{K},
\]
and the points $Q_h$'s for the pair $(\F,\M)$ are 
\[
Q_1=(1,y,y),\; Q_2=(x,1,x),\; Q_3=(0,y,0),\; Q_4=(x,0,0),\; Q_5=(1,1,1),\; Q_6=(0,0,0),
\]
and $Q_0=\M=(x,y,z)$.
In this case   $\B=\{Q_5, Q_6\}=\{(1,1,1), (0,0,0)\}$, that is $\B$ is non empty, but its cardinality is less than $2^2=4$ as there are logical dependencies ($E_1=E_2=A$). Then, to check coherence of $\M$ we cannot use the condition $\M \in \I_{\B}$, but we still need to start by checking the condition $\M \in \I_{\Q}$, which amounts to solvability of the system below.
\begin{equation*}
\left\{
\begin{array}{l}
x=\lambda_1+\lambda_2 x+\lambda_4 x+\lambda_5, \\
y=\lambda_1 y+\lambda_2+\lambda_3 y+\lambda_5, \\
z=\lambda_1 y+\lambda_2 x+\lambda_5, \\
\lambda_1+\cdots+\lambda_5=1,\;\; \lambda_h \geq 0,\; h=1,2,3,4,5,
\end{array}
\right.
\end{equation*}
where it is immediate to verify that $z \leq x$ and $z \leq y$. Moreover, the system can be written as 
\begin{equation*}
\left\{
\begin{array}{l}
xy=\lambda_1 y+\lambda_2 xy+\lambda_4 xy+\lambda_5 y, \\
xy=\lambda_1 xy+\lambda_2 x+\lambda_3 xy+\lambda_5 x, \\
z=\lambda_1 y+\lambda_2 x+\lambda_5, \\
\lambda_1+\cdots+\lambda_5=1,\;\; \lambda_h \geq 0,\; h=1,2,3,4,5.
\end{array}
\right.
\end{equation*}
By summing the first two equations, we obtain: \[
xy=z-\lambda_5 (1-x)(1-y)-\lambda_6 xy;
\] 
that is
\[
z=xy+\lambda_5 (1-x)(1-y)+\lambda_6 xy \geq xy.
\]
Then, the set of coherent extensions $z$ of $(x,y)$ to the conditional constituent $(A|H)\wedge(A|K)$ is the set $\{z : \, xy \leq z \leq \min \{x,y\}\}$ (see also \cite[Theorem 5 ]{GiSa19b}). 
\end{example}
\section{Some comparison with other approaches}
\label{SEC:COMP}
Usually in literature the notion  of conjunction has been defined as a suitable conditional event; for some of these notions
the lower and upper probability bounds  have been computed in 
 \cite{SUM2018S}.
However, by defining compound conditionals as tri-valued entities,  some basic  probabilistic properties are not satisfied.  
Within our approach  the conjunction  of conditional events
is no longer a tri-valued entity, but  it is a suitable conditional random quantity with a finite number of possible values in the unit interval. Anyway, this lack of closure does not seem a high price to pay because by our definition we preserve relevant probabilistic properties. On the other hand,  there is often a lack of closure with respect to mathematical operations. This happens, for instance, by considering the ratio of integer numbers.  

In the next subsection we make a comparison between  quasi conjunction and conjunction.
\subsection{A comparison between quasi conjunction and conjunction}
We recall  below the notion of 
	quasi conjunction (\cite{adams75}, see also \cite{Cala87,DuPr94,Scha68}),   which coincides with   Soboci\'{n}ski conjunction (\cite{CiDu12}), defined as 
	\begin{equation}\label{EQ:QC}
	\begin{array}{ll}
	Q(E_1|H_1,E_2|H_2)=[(\no{H}_1\vee E_1H_1)\wedge (\no{H}_2\vee E_2H_2)]|(H_1\vee H_2)=\\=
	(E_1H_1E_2H_2+\no{H}_1E_2H_2+\no{H}_2E_1H_1)|(H_1\vee H_2).
	\end{array}
	\end{equation}
Concerning the lower  and upper bounds on quasi conjunction,  the assessment $(x_1,x_2)$ on $\{E_1|H_1, E_2|H_2\}$, with $E_1,H_1,E_2$, $H_2$ logically independent, propagates to   the interval $[z',z'']$ on the probability of $Q(E_1|H_1,E_2|H_2)$, where 	(\cite{gilio12ijar,gilio13})
\[
z'=\max\{x_1+x_2-1,0\},\;\; \;z''=\left\{
\begin{array}{ll}
\frac{x_1+x_2-2x_1x_2}{1-x_1x_2}, &(x_1,x_2)\neq (1,1),\\
1, &(x_1,x_2)=(1,1).
\end{array}
\right.
\]
Notice that $z''\geq \min\{x_1,x_2\}$, that  is the upper bound for the quasi conjunction is  greater than or equal to the
Fr\'echet-Hoeffding upper bound.  For instance, when $x_1=x_2=\frac12$ it follows that
$z''=\frac{2}{3}>\frac{1}{2}=\min\{\frac{1}{2},\frac{1}{2}\}$.  Thus our notion of conjunction preserves  Fr\'echet-Hoeffding bounds, while quasi conjunction does not. Table  \ref{TAB:TABLEQC} illustrates the numerical values of quasi conjunction and conjunction of two conditional events. 
	\begin{table}[!h]
		\centering
		\begin{tabular}{l|l|c|c|c|c}
			&	$C_h$&  $E_1|H_1$   &  $E_2|H_2$   & $(E_1|H_1)\wedge(E_2|H_2)$ & $Q(E_1|H_1,E_2|H_2)$\\
			\hline		                 
			$C_1 $&$E_1H_1E_2H_2  $           & 1  & 1  &1  &    1  \\
			$C_2 $&		$	E_1H_1\no{E}_2H_2$    & 1  & 0 & 0 &   0   \\
			$C_3 $&		$	
			E_1H_1\no{H}_2  $       & 1  & $x_2$ &$x_2$   &    1  \\
			$C_4 $&		$	\no{E}_1H_1E_2H_2$    & 0 & 1  & 0  &   0  \\
			$C_5 $&		$	\no{E}_1H_1\no{E}_2H_2 $& 0 & 0 & 0    &   0 \\
			$C_6 $&		$	\no{E}_1H_1\no{H}_2$  & 0 & $x_2$ &0   &   0 \\
			$C_7 $&		$	\no{H}_1E_2H_2$       & $x_1$ & 1  & $x_1$  &    1  \\
			$C_8$&		$	\no{H}_1\no{E}_2H_2$  & $x_1$ & 0 & 0  &   0  \\
			$C_0 $&		$	\no{H}_1\no{H}_2 $    & $x_1$ & $x_2$ & $x_{12}$   	&   $z$ \\

		\end{tabular}
		\caption{Numerical values of the  conjunctions.  The values $x_1,x_2,x_{12},z$ denote $P(E_1|H_1)$, $P(E_2|H_2)$,  
 $\prev[(E_1|H_1)\wedge (E_2|H_2)]$ and
						$P[Q(E_1|H_1,E_2|H_2)]$, respectively. }
		\label{TAB:TABLEQC}
	\end{table}	
As shown in Table \ref{TAB:TABLEQC}, the value of the conjunction is less than or equal to the value of the quasi conjunction when $H_1\vee H_2$ is true. Then, 
by Remark \ref{REM:INEQ-CRQ},  it holds that 
\[
\prev[(E_1|H_1)\wedge (E_2|H_2)]= x_{12}\,\leq\, z=P[Q(E_1|H_1,E_2|H_2)]
\] and hence 
$(E_1|H_1)\wedge (E_2|H_2)\leq Q(E_1|H_1,E_2|H_2)$ also  when $\no{H}_1\no{H}_2$ is true. Thus, in all cases it holds that
	\begin{equation}
	(E_1|H_1)\wedge (E_2|H_2)\leq Q(E_1|H_1,E_2|H_2).
	\end{equation}
We observe that,  in the particular cases where  
 $E_1H_1\no{H}_2=\no{H}_1E_2H_2=\emptyset$, or   $x_1=x_2=1$,  
 by Theorem \ref{THM:EQ-CRQ} 
 it holds that  $z=x_{12}$ and $(E_1|H_1)\wedge (E_2|H_2)= Q(E_1|H_1,E_2|H_2)$.  More  precisely  $x_1=x_2=1$ implies $z=x_{12}=1$.   Moreover,  $E_1H_1\no{H}_2=\no{H}_1E_2H_2=\emptyset$ implies  that
  \[
  (E_1|H_1)\wedge (E_2|H_2)= Q(E_1|H_1,E_2|H_2)=E_1H_1E_2H_2|(H_1\vee H_2),
  \]
 with $z=x_{12}=P(E_1H_1E_2H_2|(H_1\vee H_2))$. In this case,  $(E_1|H_1)\wedge (E_2|H_2)$ also coincides with the  Kleene-Lukasiewicz-Heyting conjunction  $E_1H_1E_2H_2|(E_1H_1E_2H_2 \vee \no{E}_1H_1 \vee \no{E}_2H_2)$ (see \cite[Table 2]{SUM2018S}). We recall that 
the Kleene-Lukasiewicz-Heyting conjunction  coincides with the logical product between tri-events given in \cite{defi36} (see also \cite{Miln97}). In addition, we observe  that  
 \[
 Q(E_1|H_1,E_2|H_2)-(E_1|H_1)\wedge (E_2|H_2)=[(1-x_1)\no{H_1}E_2H_2+(1-x_2)E_1H_1\no{H_2}]|(H_1\vee H_2)\geq 0,
 \]
and 
 \[
 z- x_{12}=(1-x_1)P(\no{H}_1E_2H_2|(H_1\vee H_2))+
 (1-x_2)P(E_1H_1\no{H}_2|(H_1\vee H_2))\geq 0.
 \]
Then, to assess $z=x_{12}$ amounts to
\[
(1-x_1)P(\no{H}_1E_2H_2|(H_1\vee H_2))=(1-x_2)P(E_1H_1\no{H}_2|(H_1\vee H_2))=0,
\] that is 
$(P(\no{E}_1|H_1)=1-x_1=0 $ or $ P(\no{H}_1E_2H_2|(H_1\vee H_2))=0)$ and  $(P(\no{E}_2|H_2)=1-x_2=0$ or  $P(E_1H_1\no{H}_2|(H_1\vee H_2))=0)$. In addition, it is true that in a conditional bet on quasi conjunction we receive a random amount greater than or equal to the random amount received in a conditional bet on  conjunction, but in these  bets we pay two different amounts $z$ and $x_{12}$, with $z\geq x_{12}$. Moreover, $z= x_{12}$ only in  extreme cases  where some suitable conditional probabilities are zero.\\
We also recall  the notion of logical inclusion relation among  conditional events given in  \cite{GoNg88} (see also \cite{PeVi14} for an extension to conditional gambles). Given two conditional events 
$E_1|H_1$ and $E_2|H_2$, we say that $E_1|H_1$ implies $E_2|H_2$, denoted by $E_1|H_1 \subseteq E_2|H_2$, iff $E_1H_1$ {\em true} implies $E_2H_2$ {\em true} and $\no{E}_2H_2$ {\em true} implies $\no{E}_1H_1$ {\em true}; i.e., iff $E_1H_1 \subseteq E_2H_2$ and $\no{E}_2H_2 \subseteq \no{E}_1H_1$. Then, 
we remark that given two conditional events $E_1|H_1, E_2|H_2 $, with 
$E_1|H_1 \subseteq E_2|H_2$,
 for the quasi conjunction it holds that  (\cite{GiSa13IJAR})
\[
E_1|H_1\subseteq Q(E_1|H_1,E_2|H_2) \subseteq E_2|H_2,
\]
while in our approach one has
\[
(E_1|H_1)\wedge (E_2|H_2)=E_1|H_1.
\] 
Moreover,  if
 $E_1|H_1\subseteq E_2|H_2 \subseteq E_3|H_3$  then 
 \[
(E_1|H_1)\wedge (E_2|H_2)\wedge (E_3|H_3)=E_1|H_1,
 \]
 while 
 \[
 E_1|H_1\subseteq Q(E_1|H_1,E_2|H_2) \subseteq
 Q(E_1|H_1,E_2|H_2,E_3|H_3)\subseteq E_3|H_3,
 \]
 and so on (see also \cite[Theorem 9]{gilio13}). We also observe that from
 $ E_1|H_1\subseteq E_2|H_2$, it follows that 
 $P(E_1|H_1) \leq P(E_2|H_2)$ and   $E_1|H_1\leq E_2|H_2$.  This property of 
  conditional monotony of conditional probability, as shown in Remark \ref{REM:INEQ-CRQ}, holds more in general for the conditional previsions of  conditional random quantities. For instance, given $n+1$ conditional events $E_1|H_1,\cdots, E_{n+1}|H_{n+1}$,  by applying Remark \ref{REM:INEQ-CRQ} with  $X|H=\mathscr{C}_{1\cdots n+1}$ and 
  $Y|K=\mathscr{C}_{1\cdots n}$, it holds that
     $\mathscr{C}_{1\cdots n+1}\leq \mathscr{C}_{1\cdots n}$ 
     when $H_1\vee \cdots \vee H_{n+1}$ is true; then $\prev(\mathscr{C}_{1\cdots n+1})\leq \prev(\mathscr{C}_{1\cdots n+1})$ and hence      $\mathscr{C}_{1\cdots n+1}\leq \mathscr{C}_{1\cdots n}$  in all cases; a dual result is  valid for disjunctions (\cite[theorems 7 and 8]{GiSa19}). As we can see, the property of conditional monotony of conditional previsions is satisfied. 
\subsection{On Boolean algebras of conditionals}
Boolean algebras of conditionals have been  studied in \cite{FlGH15,FlGH17},  where the authors characterize the atomic structure of the algebra of conditionals and introduce the logic of Boolean conditionals. In their work the notions of conjunction $\sqcap$ 
and  disjunction $\sqcup$ 
are not (completely) specified, but it is assumed that some basic properties are satisfied. For instance, given three events $A,B,C$,  it is required  that (\cite[Proposition 1]{FlGH17}, see also \cite[Proposition 3.3]{FlGH20})
\begin{equation}\label{EQ:THIRDAXIOM}
(A|B)\sqcap (B|C)=A|C,  \mbox{ when } A\subseteq B\subseteq C.
\end{equation}
In our approach we do not start by an algebra of events, by means of  which  an  algebra of conditionals is constructed, but we  consider  arbitrary families of conditional events. Then we determine the associated constituents  and directly define the notions of conjunction and disjunction, by verifying the properties. For instance, in our approach formula (\ref{EQ:THIRDAXIOM}) holds.
Indeed, by assuming that   $A\subseteq B \subseteq C$, we obtain
\begin{equation}
(A|B) \wedge (B|C) =\left\{\begin{array}{ll}
1, &\mbox{if $A$ is true,}\\
0, &\mbox{if $\no{A}C$ is true,}\\
x_{12}, &\mbox{if $\no{C}$ is true},
\end{array}
\right.
\end{equation}		
where $x_{12}=\prev[(A|B) \wedge (B|C)]$.
Moreover, 
\begin{equation}
A|C =\left\{\begin{array}{ll}
1, &\mbox{if $A$ is true,}\\
0, &\mbox{if $\no{A}C$ is true,}\\
z, &\mbox{if $\no{C}$ is true},
\end{array}
\right.
\end{equation}		
where $z=P(A|C)$. Then, by Theorem 
 \ref{THM:EQ-CRQ}, it follows that  $x_{12}=z$ and hence  $(A|B) \wedge (B|C)=A|C$. Moreover, 
 it holds that 
\[
\prev[(A|B)\wedge(B|C)]=P(A|C)=P(AB|C)=P(A|BC)P(B|C)=P(A|B)P(B|C),
\]
which is the  well known compound probability theorem.

Our notion of conjunction  satisfies another property which is related  to the atoms of the Boolean algebra of conditionals studied in \cite{FlGH17,FlGH20}.
This property is described in the result below  (where it is not assumed that the conditioning events have positive probability).
\begin{theorem}\label{THM:ATOMS}
Let $H_1,\ldots,H_n$ be $n$ pairwise incompatible events. 
Then,\\
\begin{equation} \label{EQ:ATOMS}
 (H_1| \Omega)\wedge (H_2|\no{H}_1)\wedge \cdots \wedge (H_n|\no{H}_1\cdots \no{H}_{n-1})=
P (H_2|\no{H}_1)\cdots  P(H_n|\no{H}_1\cdots \no{H}_{n-1})\,H_1,
\end{equation}
so  that 
\[
\prev[(H_1| \Omega)\wedge (H_2|\no{H}_1)\wedge \cdots \wedge (H_n|\no{H}_1\cdots \no{H}_{n-1})]=
P(H_1)P (H_2|\no{H}_1)\cdots  P(H_n|\no{H}_1\cdots \no{H}_{n-1})
.
\]
\end{theorem}
\begin{proof}
We set $P(H_1)=x_1$ and  
$P(H_j|\no{H}_1\cdots \no{H}_{j-1})=x_{j}$, $j=2,\ldots,n$, and 
$\prev[(H_1| \Omega)\wedge (H_2|\no{H}_1)\wedge \cdots \wedge (H_n|\no{H}_1\cdots \no{H}_{n-1})]=x_{1\cdots n}$. Formula (\ref{EQ:ATOMS}) holds for $n=2$ and $n=3$. Indeed, for $n=2$ it holds that 
\begin{equation}\label{EQ:ATOMSn=2}
(H_1|\Omega) \wedge (H_2|\no{H}_1) =
\left\{\begin{array}{ll}
x_2, &\mbox{if $H_1$ is true,}\\
0, &\mbox{if $\no{H}_1$ is true,}\\
\end{array}=x_2H_1,
\right.
\end{equation}	
so that 
\begin{equation*}
\prev[(H_1|\Omega) \wedge (H_2|\no{H}_1)]=x_{12}=x_2P(H_1)=x_1x_2=P(H_1)P(H_2|\no{H}_1).
\end{equation*}	
Moreover, based on (\ref{EQ:ATOMSn=2}) and on Definition \ref{CONJUNCTION}, for $n=3$ we obtain
\begin{equation*}
\begin{array}{ll}
(H_1|\Omega)\wedge  (H_2|\no{H}_1)\wedge (H_3|\no{H}_1\no{H}_2) =x_2H_1\wedge (H_3|\no{H}_1\no{H}_2)=\\
=x_2[(H_1H_3\no{H}_2\no{H}_1+x_{1}\no{\Omega}H_3\no{H}_2\no{H}_1+x_{3}(H_1\vee H_2)H_1  ]|(\Omega \vee \no{H}_1\no{H}_2 )=x_2(x_3H_1|\Omega)=x_2x_3H_1.
\end{array}
\end{equation*}	
We assume by induction that (\ref{EQ:ATOMS}) holds for $n-1$, that is 
\begin{equation}\label{EQ:INDATOMS}
(H_1| \Omega)\wedge (H_2|\no{H}_1)\cdots \wedge (H_{n-1}|\no{H}_1\cdots \no{H}_{n-2})=x_{2}\cdots x_{n-1}H_1,
\end{equation}
then  we prove that it holds for $n$.
Indeed, from (\ref{EQ:INDATOMS}) we
obtain 
\[
 (H_1| \Omega)\wedge (H_2|\no{H}_1)\wedge \cdots \wedge (H_n|\no{H}_1\cdots \no{H}_{n-1})=x_{2}\cdots x_{n-1}H_1\wedge (H_n|\no{H}_1\cdots \no{H}_{n-1}).
\]
Moreover, by Definition \ref{CONJUNCTION}, it holds that
\[
\begin{array}{ll}
H_1\wedge (H_n|\no{H}_1\cdots \no{H}_{n-1})=\\
=(H_1H_n\no{H}_{1}\cdots \no{H}_{n-1}+x_{1}\no{\Omega}H_n\no{H}_1\cdots \no{H}_{n-1}+x_{n}(H_1\vee \cdots \vee H_{n-1})H_1  )|(\Omega \vee \no{H_1}\no{H}_2 )=\\
=x_{n}H_1|\Omega=x_{n}H_1.
 \end{array}
\]
Finally,
\[
\begin{array}{ll}
(H_1| \Omega)\wedge (H_2|\no{H}_1)\wedge \cdots \wedge (H_n|\no{H}_1\cdots \no{H}_{n-1})=x_{2}\cdots x_{n-1}x_nH_1,
 \end{array}
\]
and hence $x_{1\cdots n}=x_1\cdots x_n$.
\end{proof}
\subsection{Some theoretical aspects and applications of conjunction}
In this section we recall some theoretical aspects and applications of our approach to compound conditionals.\\
- All the basic properties valid for the unconditional events are satisfied in our theory of compound conditionals. For instance, (generalized) De Morgan’s Laws are satisfied; moreover  the formula $P(E_1\vee E_2)=P(E_1)+P(E_2)-P(E_1E_2)$ becomes 
$\prev[(E_1|H_1)\vee(E_2|H_2)] =P(E_1|H_1)+P(E_2|H_2)-\prev[(E_1|H_1)\wedge(E_2|H_2)]$. \\
-The Fr\'echet-Hoeffding lower and upper prevision bounds for the conjunction (and for the disjunction) of two conditional events still hold.
\\
- A generalized  inclusion-exclusion formula for the disjunction of conditional events  holds in our approach to compound conditionals.\\
- We can introduce the notion of conditional constituents, with  properties analogous to the case of unconditional events, which allow to characterize coherence when the basic events are logically independent.\\
- Conjoined conditionals have been  applied to probabilistic nonmonotonic reasoning (\cite{GiPS20,GiSa19}), by obtaining a characterization for the property of probabilistic entailment of Adams (\cite{adams75}).
In particular, in \cite{GiSa19} it has been shown that a conditional event $E_{n+1}|H_{n+1}$ is p-entailed from  a p-consistent family of $n$ conditional events $E_1|H_1,\cdots,  E_n|H_n$ if and only if the conjunction  $\mathscr{C}_{1\cdots n+1}
 $ of  the premises and the conclusion
 coincides with the conjunction 
 $\mathscr{C}_{1\cdots n}$
 of the  premises. Another equivalent condition is that $\mathscr{C}_{1\cdots n}\leq E_{n+1}|H_{n+1}$. Moreover, by exploiting a suitable notion of iterated conditional, in \cite{GiPS20} it has been shown that a family $\{E_1|H_1,E_2|H_2\}$ p-entails a conditional event $E_3|H_3$ if and only if the iterated conditional $(E_3|H_3)|((E_1|H_1)\wedge (E_2|H_2))$ is constant and coincides with 1.\\
- Compound conditionals have been also applied to the psychology of the probabilistic reasoning, where by exploiting the notion of iterated conditional, the probabilistic modus ponens has been generalized to conditional events  (\cite{SaPG17}).\\
- Another application  to  one-premise and two-premise centering inferences 
has been given in 
\cite{GOPS16,SPOG18},  by also determining the lower and upper prevision bounds for the conclusion of the rules. \\
-We remark that, like in \cite{adams75,Kauf09} and differently from \cite{McGe89}, the Import-Export Principle is not valid in our theory of compound and iterated conditionals. Then, 
as proved in \cite{GiSa14} (see also \cite{SGOP20,SPOG18}), we avoid Lewis’ triviality results (\cite{lewis76}).  In addition,  within our theory,
we can    explain  some intuitive probabilistic assessments
discussed in  \cite{douven11b},   by suitably formalizing different kinds of latent information  (\cite{SGOP20}).

\section{Conclusions}
\label{SEC:CONC}
In this paper we  deepened  the study of conjunctions 
and disjunctions among conditional events in the framework of conditional random quantities. 
We proved that the Fr\'echet-Hoeffding bounds are a necessary coherence condition for the prevision assessments on  $\{\mathscr{C}_{1\cdots k},\mathscr{C}_{k+1\cdots n},\mathscr{C}_{1\cdots n}\}$, for every  $1\leq k\leq n-1$. 
We obtained 
a decomposition formula for the conjunction and  we introduced the set of (non negative) conditional constituents  $\mathscr{K}$ for a  family $\mathscr{E}$ of $n$ conditional events. 

We showed that, as in the case of unconditional events, the sum of the conditional constituents is equal to 1 and for each pair of them the conjunction is equal to 0. 
We verified that, for each non empty subset $S$, the conjunction  $\C_{S}$  is the sum of suitable conditional constituents in $\mathscr{K}$ and  hence  the prevision of $\C_S$ is the sum of the previsions of such conditional constituents.

We obtained a generalized  inclusion-exclusion formula for the disjunction of $n$ conditional events;  we proved a suitable  distributivity property and we examined some related  probabilistic results. 

Under logical independence,  we characterized in terms of a suitable convex hull the set of all coherent prevision assessments on a  family $\mathscr{F}$ containing $n$ conditional events and all the possible conjunctions among them. We showed that such a  characterization  amounts to the  solvability of a linear system and  we described the set of all coherent prevision assessments on  $\mathscr{F}$ by  a list of linear inequalities.
Based on the $(2^{n}-1)$-dimensional simplex $\Delta$, we characterized  (still under logical independence) the set of all coherent prevision assessments on $\mathscr{F} \cup \mathscr{K}$.

Then, given a coherent assessment $\mathcal{M}$ on $\mathscr{F}$, we showed that every  possible value $Q_h$ of the random vector associated with $\mathscr{F}$ is itself a particular coherent assessment on $\mathscr{F}$.
We deepened some aspects of coherence by illustrating  examples and counterexamples. 

We made  a comparison with other approaches, by obtaining  a result related to the notion of atom of a Boolean algebra of conditionals  introduced in \cite{FlGH17,FlGH20}. 
Finally,  we discussed the significance and perspectives of our theory by illustrating basic theoretical aspects and some applications to nonmonotonic reasoning and to the psychology of  probabilistic reasoning. 

Future work should concern in particular the study of necessary and sufficient conditions of coherence in the general case of logical dependencies among the basic unconditional events, by exploiting the set of conditional constituents.

Further future work could concern the study of compound conditionals in the setting of imprecise probabilities and gambles. Indeed, indicators of conditional events are ternary gambles and our conjunction builds $n$-ary gambles from ternary ones. Another interesting aspect that could be deepened is the study of the role of our compound conditionals in the framework of fuzzy logic and information fusion (see, e.g., \cite{CoPV17,CoVa18,DuFP20,dubo16}). 
\section*{Declaration of competing interest}
We wish to confirm that there are no known conflicts of interest associated with this publication and there has been no significant financial support for this work that could have influenced its outcome.

\section*{Acknowledgments}
\noindent
We thank the three anonymous reviewers for their careful reading of our   manuscript. Their many insightful comments and suggestions  were very helpful in improving this paper.  
Giuseppe Sanfilippo has been partially supported by the INdAM–GNAMPA Project
(2020 Grant U-UFMBAZ-2020-000819).
\bibliographystyle{model2-names} 
\bibliography{ijar_PerugiaR2}
\end{document}